\documentclass[11pt]{amsart}

\usepackage{amsmath,amsthm, amscd, amssymb, amsfonts, times}
\usepackage{epsfig,pb-diagram}
\input xy
\xyoption{all}
\usepackage[usenames, dvipsnames]{xcolor}

\newcommand{\Hc}{{\mathcal H}}
\newcommand{\HH}{{\mathcal H}}

\newcommand{\li}[1]{\overline{#1}}

\newcommand{\dG}{{\underline{\underline{G}}}}

\newcommand{\otb}{{\overline{\otimes}}}

\newcommand{\Mo}{{\mathcal M}}
\newcommand{\No}{{\mathcal N}}
\newcommand{\Cc}{{\mathcal C}}

\newcommand{\bca}{\overline{\mathcal D}}

\newcommand{\ot}{{\otimes}}

\newcommand{\Zc}{{\mathcal Z}}

\newcommand{\ca}{{\mathcal C}}

\newcommand{\Do}{{\mathcal D}}

\newcommand{\Bc}{{\mathcal B}}

\newcommand{\Fc}{{\mathcal F}}
\newcommand{\Gc}{{\mathcal G}}

\newcommand{\opp}{\rm{op}}

\newcommand{\camod}{{}_{\ca}\rm{Mod}}

\newcommand{\bb}{{\mathcal B}}

\newcommand{\ku}{{\Bbbk}}

\newcommand{\uno}{ \mathbf{1}}

\newcommand{\id}{\mbox{\rm id\,}}
\newcommand{\Id}{\mbox{\rm Id\,}}
\newcommand{\idn}{\mbox{\rm id\,}}
\newcommand{\doscat}{\textbf{2Cat}}

\newcommand{\Obj}{\mbox{\rm Obj\,}}
\newcommand{\Mod}{\mbox{\rm Mod\,}}

\newcommand{\Fun}{\operatorname{Fun}}
\newcommand{\Pse}{\textbf{2Cat}}

\newcommand{\unoli}[1]{1_{\overline{#1}}}
\newcommand{\unos}[1]{1_{\widetilde{#1}}}

\newcommand{\s}[1]{\widetilde{#1}}

\theoremstyle{plain}

\numberwithin{equation}{section}

\newtheorem{teo}{Theorem}[section]

\newtheorem{lema}[teo]{Lemma}

\newtheorem{cor}[teo]{Corollary}

\newtheorem{prop}[teo]{Proposition}

\newtheorem{claim}{Claim}[section]

\theoremstyle{definition}

\newtheorem{defi}[teo]{Definition}

  \newtheorem{exa}[teo]{Example}

\theoremstyle{remark}

\newtheorem{rmk}[teo]{Remark}

\def\pf{\begin{proof}}

\def\epf{\end{proof}}

\theoremstyle{remark}

\newcounter{commentcounter}

\begin{document}

\title[Group actions on 2-categories]{Group actions on 2-categories}
\author[ Bernaschini , Galindo, Mombelli ]{
Eugenia Bernaschini, C\'esar Galindo and Mart\'\i n Mombelli
}
\keywords{tensor category; module category; bicategory}
\address{Facultad de Matem\'atica, Astronom\'\i a y F\'\i sica
\newline \indent
Universidad Nacional de C\'ordoba
\newline
\indent CIEM -- CONICET
\newline \indent Medina Allende s/n
\newline
\indent (5000) Ciudad Universitaria, C\'ordoba, Argentina} \email{ m.e.bernaschini@gmail.com }
 \email{mombelli@mate.uncor.edu
\newline \indent \emph{URL:}\/ http://www.famaf.unc.edu.ar/$\sim$mombelli}
\address{ Departamento de Matem\'aticas, Universidad de los Andes, 
\newline \indent
Carrera 1 N. 18A - 10, Bogot\'a, Colombia }
\email{cn.galindo1116@uniandes.edu.co}
\begin{abstract}
We study  actions of discrete groups on 2-categories. The motivating examples are actions on the 2-category of representations of finite tensor categories and their relation with the extension theory of tensor categories by groups. Associated to a group action on a 2-category, we construct the 2-category of equivariant objects. We also introduce the $G$-equivariant notions of pseudofunctor, pseudonatural transformation and modification. Our first main result is a coherence theorem for 2-categories with an action of a group. For a 2-category $\bb$ with an action of a group $G$, we construct a braided $G$-crossed monoidal category $\mathcal{Z}_G(\bb)$ with trivial component the Drinfeld center of $\bb$. We  prove that, in the case of a $G$-action on the 2-category of representation of a tensor category $\ca$, the 2-category of equivariant objects is biequivalent to the module categories over an  associated $G$-extension of $\ca$. Finally, we prove that the center of the equivariant 2-category is monoidally equivalent to the equivariantization of a relative center, generalizing  results obtained in \cite{GNN}.

\end{abstract}

\subjclass[2000]{18D05, 18D10}

\date{\today}
\maketitle

\section*{Introduction}

The theory of 2-categories appears in a natural way in diverse contexts.
For example, it was used by Rouquier to ``categorify'' certain algebraic objects  \cite{R} and appears in topological field
theories \cite{FSV}, \cite{NS}. The theory of representations of 2-categories has been initiated 
in a series of papers \cite{MM1, MM2, MM3}.

Our motivation for the study of 2-categories comes from  the theory of tensor categories.
For a tensor category $\ca$, a representation of $\ca$, or $\ca$-module category, is a category $\Mo$
equipped with an associative action $\ca\times \Mo\to \Mo$ satisfying certain conditions. Given two
$\ca$-module categories $\Mo, \No$, the category $\Fun_\ca(\Mo,\No)$ is the category whose objects are
 $\ca$-module functor between $\Mo$ and $\No$, and morphisms are  $\ca$-module natural transformations. The 2-category of (left) 
$\ca$-modules $\camod$ has as 0-cells $\ca$-module categories, 1-cells $\ca$-module functors between them and
2-cells are $\ca$-module natural transformations. This 2-category is a strong invariant of the tensor category $\ca$.

Given a 2-category $\Bc$ and a 2-monad  $T:\Bc\to \Bc$ on $\Bc$, in \cite{MN}, the notion of the  \textit{equivariantization} 2-category $\Bc^T$ was presented. The equivariantization of a 2-category by a group was  studied later in \cite{HSV}.
 
One of the purposes of the paper is to explicitly describe an action of a group $G$ on a 2-category $\Bc$, and describe all ingredients of the resulting equivariantization 2-category $\Bc^G$. An action of a group $G$ on a 2-category $\Bc$ consists of
\begin{itemize}
 \item a family of pseudofunctors $F_g:\Bc\to \Bc$, $g\in G$,
 \item pseudonatural equivalences $\chi_{g,h}: F_g\circ F_h\to F_{gh}$, 
 \item invertible modifications 
 $$\omega_{g,h,f}: \chi_{gh,f} \circ (\chi_{g,h} \ot  \idn_{F_f }) \Rightarrow \chi_{g,hf}\circ (\idn_{F_g} \ot\chi_{h,f}), $$
\end{itemize}
for any $g, h, f\in G$, satisfying certain axioms.
We also prove a coherence theorem for group action, stating that there exists another equivalent action of $G$ on $\Bc$, such that all pseudofunctors $F_g$ involved in the group action are 2-functors, $ F_g\circ F_h= F_{gh}$, and $\chi_{g,h}$, $\omega_{g,h,f}$ are all the identity. As an application of the coherent theorem we prove that associated to every action of group $G$ on a 2-category $\bb$ there is a braided $G$-crossed monoidal category $\mathcal{Z}_G(\bb)$ such that the trivial component is $\mathcal{Z}(\bb)$, the Drinfeld center of $\bb$. 

An important example comes from the theory of tensor categories. We show that, if $\Do=\oplus_{g\in G} \Do_g$ is a $G$-graded tensor category, and $\Do_1=\ca$, there is an action of the group $G$ acts on ${}_{\ca}\Mod$, the 2-category of representations of $\ca$, and there is a biequivalence
$$({}_{\ca}\Mod^{\opp})^G \simeq {}_{\Do}\Mod.$$
The coherence theorem for group actions allows us to construct an associated strict braided crossed monoidal category  and to prove that there is a monoidal equivalence  between the center $\Zc(\Bc^G)$ of the equivariantization and the monoidal category of pseudonatural transformations of the forgetful pseudofunctor $\Phi:\Bc^G\to \Bc$. When applied this result to the 2-category $({}_{\ca}\Mod)^G $, we recover the results from \cite{GNN}, on the center of graded tensor categories.
\medbreak

The contents of the paper are organized as follows.  In Section \ref{Section:2-cat} we recall the basics of 2-categories. For any pseudofunctor $\Hc:\Bc\to \Bc'$ we define the monoidal category $\Zc(\Hc)$ of pseudonatural transformations $\eta:\Hc\to \Hc$. When $\Hc$ is the identity pseudofunctor, $\Zc(\Id)$ is a braided monoidal category  called the \textit{center} of the 2-category.

In Section \ref{Section:group-2-cat} we explicitly describe the notion of a group action on a 2-category. Given two 2-categories $\Bc,\Bc'$ equipped with an action of a group $G$, we define the notion of $G$-pseudofunctor between them. When a $G$-pseudofunctor is a biequivalence, we say that  $\Bc,\Bc'$ are $G$-biequivalent. Also, we define the notions of $G$-pseudonatural transformation and $G$-modifications. All these data, turns out to be a 2-category, denoted by $\Pse^G(\Bc,\Bc')$. 
The equivariant 2-category is $\Bc^G=\Pse^G(\mathcal{I},\Bc)$, where $\mathcal{I}$ is the unit 2-category, where $G$ acts trivially. 

In Section \ref{Section:coherence} we prove that any 2-category with a group action is $G$-biequivalent to another one where the action is \textit{strict}. Section \ref{Section:equivariant} is devoted to explicitly describe all ingredients in the equivariant 2-category $\Bc^G$.

In Section \ref{Section:examples-from-cat} we show an example coming from graded tensor categories. If $\Do=\oplus_{g\in G} \Do_g$ is a $G$-graded tensor category, then the group $G$ acts on the 2-category ${}_{\Do_1}\Mod$ of left $\Do_1$-modules. The resulting equivariant 2-category 
$({}_{\Do_1}\Mod)^G$ is biequivalent to ${}_{\Do}\Mod$. In Section \ref{braided-Gcrossed} we define the $G$-braided center of a 2-category with an action of a group $G$.
In Section \ref{Section-center}, we show that there is a monoidal equivalence $\Zc( \Bc^G)\simeq \Zc(\Phi)^G, $ where $\Phi:\Bc^G\to \Bc$ is the forgetful pseudofunctor. When applied to the example $({}_{\ca}\Mod)^G$, we recover results from \cite{GNN}.

\subsection*{Acknowledgments} The  work of E.B and M.M. was  partially supported by CONICET, Secyt (UNC),  Argentina. M.M.
is grateful to the department of mathematics at Universidad de los Andes, Bogot\'a, where part of this work was done, for the kind hospitality.

\section{2-categories}\label{Section:2-cat}
 
Let us briefly recall the notion of a 2-category. For more details,
the reader is referred to \cite{KSt,St}.
For any 2-category $\Bc$, the set of objects, also called \textit{0-cells}, will be denoted by $\operatorname{Obj}(\Bc)$. The composition in each hom-category $\Bc(A, B)$, that is, the vertical
composition of 2-cells, is denoted by juxtaposition $fg$, while the symbol $\circ$ is used to denote the horizontal
composition functors $$\circ : \Bc(B, C) \times \Bc(A, B) \to \Bc(A, C).$$ The identity of a 0-cell $A$ is written
as $I_A : A \to A$. For any 1-cell $X$ the identity will be denoted $\id_X$ or sometimes simply as $1_X$, when space saving is needed. For any 2-category $\Bc$, we shall denote by $\Bc^{\opp}$ the 2-category that is obtained from $\Bc$ by reversing 1-cells.

  \begin{exa}\label{initial-2cat} The unit 2-category $\mathcal{I}$ has a single 0-cell, named $\star$. The monoidal category $\mathcal{I}(\star,\star)$ is the unit monoidal category. 
\end{exa}

A \textit{pseudofunctor}  $(F,\alpha):\Bc\to \Bc'$, consists of a function $F:\operatorname{Obj}(\Bc)\to \operatorname{Obj}(\Bc')$, a family of functors $F:\Bc(A,B)\to \Bc'(F(A),F(B)),$ for each $A, B\in \operatorname{Obj}(\Bc)$, a collection of isomorphisms 
$\phi_A: I_{F(A)}\to F(I_A),$
and a family of natural isomorphisms
\[\xymatrix{
\Bc(B,C)\times \Bc(A,B) \ar[d]_{F\times F} \ar[rr]^{\circ} & {\ar @{} [d] |{\Uparrow \alpha}}& \Bc(A,C) \ar[d]^{F} \\
\Bc'(F(B),F(C))\times \Bc'(F(A),F(B)) \ar[rr]^-{\circ},&& \Bc'(F(A),F(C))
}\] for 0-cells $A, B, C$, subject to the usual axioms. A pseudofunctor is called \textit{unital} if $F(I_A)= I_{F(A)}$, for any 0-cell $A$, and the isomorphisms $\phi_A$ are the identities. A pseudofunctor is called
 a \textit{2-functor} if the associativity isomorphisms $\alpha$ are the identities.

If $F$, $G$ are pseudofunctors, a \textit{pseudonatural transformation} 
\xymatrix{
\bb {\ar @{} [r] |{\downarrow \chi}} \ar@/^1pc/[r]^F \ar@/_1pc/[r]_G& \bb'
}
consists of a family of 1-cells $\chi^0_A: F(A)\to G(A)$, $A\in \operatorname{Obj}(\Bc)$ and isomorphisms

$$\xymatrix{
F(A) \ar[r]^-{F(X)} \ar[d]_{\chi^0_A} & F(B) \ar[d]^{\chi^0_B}\\
G(A)\ar[r]_{G(X)}  {\ar @{} [ru] |{\Downarrow \chi_X}} & G(B)
}$$natural in $X\in \Bc(A,B)$, subject to the usual axioms. If $\chi, \theta$ are pseudonatural transformations, a \textit{modification} from 
\xymatrix{
\bb {\ar @{} [r] |{\downarrow \chi}} \ar@/^1pc/[r]^F \ar@/_1pc/[r]_G& \bb'
} to 
\xymatrix{
\bb {\ar @{} [r] |{\downarrow \theta}} \ar@/^1pc/[r]^F \ar@/_1pc/[r]_G& \bb'
}, consists of a family of 2-cells $\omega_A:\chi^0_A\to \theta^0_A$,  such that the diagrams
$$
\xymatrix{
\chi^0_B\circ F(X) \ar[d]_{\omega_B\circ \operatorname{id}_{F(X)}} \ar[r]^{\chi_X} &  G(X)\circ \chi^0_A \ar[d]^{ \operatorname{id}_{G(X)}\circ \omega_A}\\
\theta^0_B\circ F(X) \ar[r]^{\theta_X} & G(X)\circ \theta^0_A
}
$$
commute for all $A\in \operatorname{Obj}(\Bc)$. This modification will be denoted as $\omega:\chi\Rightarrow \theta$.  Given pseudofunctors $F,G:\bb\to \bb$, we shall denote $\operatorname{Pseu-Nat}(F,G)$ the category where objects are pseudonatural transformations from $F$ to $G$ and arrows are modifications.

A 1-cell $X\in \Bc(A,B)$ is called an \textit{equivalence} if there exists a 1-cell $Y\in \Bc(B,A)$ such that $X\circ Y\cong I_B$ and $Y\circ X\cong I_A$. We will say that an invertible 1-cell $X$ is an \textit{isomorphism} if there is $X^{*}\in \Bc(B,A)$ such that  $X\circ X^{*}= I_B$ and $X^{*}\circ X= I_A$. The next result will be useful later to simplify some proofs.

\begin{prop}\label{equivalence-1cells-iso}
Every 2-category (or bicategory) is biequivalent to a 2-category where every equivalence 1-cell is an isomorphism.
\end{prop}
\begin{proof}
The proof goes along the lines of \cite[Theorem 1.4]{GPS}. Since every category  is equivalent to a skeletal one. Every bicategory $\bb$ is biequivalent to a locally skeletal one $\bb'$, that is, each of its hom-category is skeletal.  Then in $\bb'$, every 1-cell equivalence is an isomorphism.   
By  Street's Yoneda lemma for bicategories \cite[p.117 ]{St2}, the Yoneda embedding $$\bb'\to \textbf{Bicat}(\bb',\textbf{Cat}): A\mapsto \bb'^{\opp}(A,-),$$ is locally an equivalence. Therefore, $\bb'$ is biequivalent to $\bb''$; the full sub-2category of $\textbf{Bicat}(\bb'^{\operatorname{op}},\textbf{Cat})$ determined by the contravariant representables. Since every equivalence in $\bb'$ is an isomorphism, every equivalence in $\bb''$ is an isomorphism and $\bb$ is biequivalent to $\bb''$.   
\end{proof}

\subsection{The tricategory of 2-categories}

Given a pair of 2-categories $\bb$ and $\bb'$, we can define the \textit{functor 2-category},  \doscat$(\bb,\bb')$, whose 0-cells are pseudofunctors  $\bb\to \bb'$, whose 1-cells are pseudonatural transformations, and whose 2-cells are modifications.
Given 2-categories $\bb, \bb'$ and  $\bb''$, we define a pseudo-functor $$\otimes: \doscat(\bb',\bb'')\times \doscat(\bb,\bb')\to \doscat(\bb,\bb''),$$ called the \textit{tensor product}.  The tensor product at the level of pseudofunctors is  the composition. The tensor product of
pseudonatural transformations is 
\begin{equation}\label{tensorp-psnat}
(\xymatrix@C=8pt{ \Bc' \ar@/^0.6pc/[rr]^{G}\ar@{}[rr]|{\downarrow \beta}
\ar@/_0.6pc/[rr]_{ G'} &  & \Bc''}\big)\big(\xymatrix@C=8pt{ \Bc
\ar@/^0.6pc/[rr]^{F}\ar@{}[rr]|{\downarrow \alpha}
\ar@/_0.6pc/[rr]_{ F'} &  & \Bc'}\big)=\big(\xymatrix@C=9pt{ \Bc
\ar@/^0.7pc/[rr]^{GF}\ar@{}[rr]|{\downarrow \beta\otimes \alpha}
\ar@/_0.7pc/[rr]_{ G'\!F'} &  & \Bc''}\big),
\end{equation}
 where 
\begin{align*}
    (\beta \otimes \alpha)_A&=\beta_{F'(A)}\circ G(\alpha_A)\\
    (\beta \otimes \alpha)_X&=(\beta_{F'(X)}\circ \operatorname{id}_{G(\alpha^0_A)})(\operatorname{id}_{\beta^0_{F'(B)}}\circ G(\alpha_X)).
\end{align*} 
 Here, the isomorphisms constraints of the pseudofunctors have been omitted as a space-saving measure. If $\beta':G\to G'$ and $\alpha':F\to F'$ are another pseudonatural transformations and
 $\omega:\beta\to \beta'$ and  $\omega':\alpha\to \alpha'$ are modifications, their tensor product is defined as $\omega \otimes \omega': \beta\ot \alpha\to \beta'\ot \alpha'$, $(\omega \otimes \omega')_A := \omega_{F'(A)}\circ G(\omega'_A)$, for any 0-cell $A$.
\smallbreak
If  $\alpha:F\to F'$ and $\beta: H\to H'$ are pseudonatural transformations between pseudofunctors  $F,F'\in \doscat(\Bc',\Bc'') ,H,H'\in \doscat(\bb,\bb')$, then there is a modification

\[\xymatrix{
&F'H \ar[rd]^{\id_{F'}\otimes \beta}&\\
FH {\ar @{} [rr] |{\Downarrow c_{\alpha,\beta}}} \ar[rd]_{ \id_H\otimes \beta} \ar[ru]^{\alpha\otimes \id_H}&& F'H'\\
&FH' \ar[ru]_{\alpha\otimes \id_{H'}}&
}\]
given by 
\begin{equation}\label{comparison-const}
(c_{\alpha,\beta})_A:=\alpha_{\beta_A}^{-1}: F'(\beta_A)\circ \alpha_{H(A)}\to \alpha_{H'(A)}\circ F(\beta_A).
\end{equation}
This modification is called the \textit{comparison constraint.}

The tensor product is associative only at the level of pseudofunctors, but not for pseudonatural transformations. There exists an associativity constraint

\[
\xymatrix{
KHG {\ar @{} [rr] |{\Downarrow a_{\alpha,\beta,\gamma}}} \ar@/^2pc/[rr]^{(\alpha \otimes \beta)\otimes \gamma} \ar@/_2pc/[rr]_{\alpha\otimes (\beta\otimes \gamma)}&& K'H'G'
}
\]
for pseudonatural transformations $\alpha:K\to K'$, $\beta:H\to H'$ and $\gamma:G\to G'$.
The modification 
$$(a_{\alpha,\beta,\gamma})_A:\alpha_{F'H'(A)}\circ G(\beta_{H'(A)})\circ GF(\gamma_A)\to\alpha_{F'H'(A)}\circ G(\beta_H'(A)\circ F(\gamma_A)) $$ 
is defined by $(a_{\alpha,\beta,\gamma})_A=\operatorname{id}_{\alpha_{F'H'(A)}}\circ G_2(\beta_{H'(A)},F(\gamma_A))$. It is easy to see that $a$ satisfies the pentagonal identity.

\subsection{Finite tensor categories}\label{subsection:tensorcat}

A (strict) monoidal category is a 2-category with one single 0-cell. A \emph{finite tensor category over} $\ku$ is a finite $\ku$-linear abelian rigid monoidal
category $\ca$  such that the tensor product functor $\otimes : \ca \times \ca
\to \ca$ is $\ku$-linear in each variable. The reader is referred to \cite{eo}.

Suppose $\ca$ and $\Do$ are strict tensor categories. A \emph{monoidal functor} $(F, \xi, \phi):\ca\to 
\Do$ is a pseudofunctor between the corresponding 2-categories. Explicitly, it consists of a functor $F:\ca\to \Do$, natural isomorphisms 
$\xi_{X,Y}:F(X)\ot F(Y)\to F(X\ot Y),$ $X, Y\in \ca$, and isomorphism $\phi:\uno\to
F(\uno)$, satisfying certain axioms. If $(F,\xi,\phi), (F',\xi',\phi')$ are monoidal functors , a {\it
 natural monoidal transformation} $\theta:(F,\xi,\phi)\to
(F',\xi',\phi')$ is a natural transformation $\theta:F\to
F'$, such that for any pair of objects $X, Y$
\begin{equation}\label{tensort1} \theta_{\uno}\phi=\phi',\quad
\theta_{X\ot Y}\xi_{X,Y}=\xi'_{X,Y} (\theta_{X}\ot \theta_Y).
\end{equation}

\subsection{The endomorphism category of a pseudofunctor}\label{relative-center}

If $\bb$ is a 2-category, the  monoidal category $$\Zc(\bb)=\doscat(\bb,\bb)(\Id_{\bb},\Id_{\bb})$$ is exactly the center of $\bb$, \textit{i.e.}, the obvious generalization of the center construction of a  monoidal category. See \cite{MS}.
\medbreak

Let $\Bc, \Bc'$ be two 2-categories and 
$(\Hc, \alpha):  \Bc\to\Bc' $ be a unital pseudofunctor.
Denote $\Zc(\Hc)=\doscat(\bb,\bb')(\Hc,\Hc)$; the category of pseudonatural transformations of the pseudofunctor $\Hc$. This is a monoidal category with tensor product described in the previous section. Explicitly, objects in    $\Zc(\Hc)$ are pairs $(V,\sigma)$, where
 $$V=\{V_A\in \Bc'(\Hc(A),\Hc(A)) \text{1-cells, for any} A\in \Bc \}, $$   
 $$\sigma=\{\sigma_X: V_B\circ \Hc_{A,B}(X)\to \Hc_{A,B}(X)\circ V_A\},$$
 where, for any $X\in \Bc(A,B)$, $\sigma_X$ is a natural isomorphism 
 2-cell such that
\begin{equation}\label{relativ-cent1} \sigma_{I_A}=\id_{V_A},
(\alpha_{X,Y}\circ\id_{V_A}) \sigma_{X\circ Y}= (\id_{\Hc(X)}\circ \sigma_Y)(\sigma_X\circ \id_{\Hc(Y)})(\id_{V_B}\circ \alpha_{X,Y}),
\end{equation}
 for any 0-cells $A,B, C\in \Bc$, and any pair of 1-cells
 $X\in \Bc(C,B)$, $Y\in \Bc(A,C)$.

If $(V,\sigma)$, $(W,\tau)$ are two objects in $\Zc(\Hc)$, a morphism $f:(V,\sigma)\to (W,\tau)$ in $\Zc(\Hc)$ is a collection of 2-cells $f_A:V_A\Rightarrow W_A$, $A\in\Bc$ such that
\begin{equation}\label{relativ-cent2} 
(\id_{\Hc(X)}\circ f_A)\sigma_X=\tau_X (f_B\circ \id_{\Hc(X)}),
\end{equation}
for any 1-cell $X\in\Bc(A,B)$. The category $\Zc(\Hc)$ has a monoidal product defined as follows. Let $(V,\sigma), (W,\tau)\in \Zc(\Hc)$ be two objects. Then $(V,\sigma)\ot (W,\tau)=(V\ot W, \sigma\ot \tau)$, where for any 0-cells 
$A, B\in \Bc$, and $X\in \Bc(A,B)$
\begin{equation}\label{tensor-prod-relativ} (V\ot W)_A=V_A\circ W_A,\quad
(\sigma\ot \tau)_X = (\sigma_X\circ \id_{W_A})(\id_{V_B}\circ \tau_X).
\end{equation}
If $(V,\sigma), (V',\sigma'), (W,\tau), (W',\tau')\in \Zc(\Hc)$ are objects, and $f:(V,\sigma)\to (V',\sigma'),$ $f':(W,\tau), (W',\tau')$ are morphisms in $\Zc(\Hc)$, then $f\ot f':(V,\sigma)\ot (V',\sigma') \to (W,\tau)\ot (W',\tau')$ is defined by
$$(f\ot f')_A=f_A\circ f'_A, $$
for any 0-cell $A$. The unit $(\uno,\iota)\in\Zc(\Hc)$ is the object
$$ \uno_A=I_A, \quad \iota_X=\id_X,$$
for any 0-cells $A, B$ and any 1-cell $X\in \Bc(A,B).$
The center  $\Zc(\Id_\Bc)$ of the identity pseudofunctor $\Id_\Bc:\Bc\to\Bc$ is denoted as $\Zc(\Bc)$, and it coincides with the definition presented in \cite{MS}.

\section{Group actions on 2-categories}\label{Section:group-2-cat}
Assume $G$ is a group and $\Bc$ is a 2-category. We shall denote by $\dG$ the 2-category that has 0-cells the elements 
of the group $G$. For any pair $g,h\in G$
$$ \dG(g,h)=\begin{cases} \text{ the unit category, if}\quad g=h\\
\emptyset\,  \text{ if }\quad g\neq h.
  \end{cases}
$$
Moreover, $\dG$ is a monoidal 2-category, see \cite{GPS}. Since $\Pse(\Bc, \Bc)$ is also a monoidal 2-category, we define
an \emph{action} of $G$ on $\Bc$ as 
a weak monoidal homomorphism $(\Fc,\chi,\omega,\iota,\kappa,\zeta): \dG \to \Pse(\Bc, \Bc)$. See for example \cite{GPS}.

\medbreak

Explicitly, an action of $G$ on a 2-category $\Bc$ consists of the following data:
\begin{itemize}
 \item A family of pseudofunctors $F_g:\Bc\to \Bc$, $g\in G$,
 \item pseudonatural equivalences $(\chi_{g,h}, \chi^0_{g,h}): F_g\circ F_h\to F_{gh}$, $g, h\in G$,
 \item a pseudonatural equivalence $\iota: \Id_\Bc\to F_1$,
 
 \item for any $g,h,f\in G$ invertible modifications 
 $$\omega_{g,h,f}: \chi_{gh,f} \circ (\chi_{g,h} \ot  \idn_{F_f }) \Rightarrow \chi_{g,hf}\circ (\idn_{F_g} \ot\chi_{h,f}), $$
 $$\kappa_g: \chi_{1,g} \circ (\iota\ot \idn_{F_g} ) \Rightarrow \id_{F_g}, \quad \zeta_g: \chi_{g,1}\circ (\idn_{F_g}\ot\iota)\Rightarrow
 \id_{F_g},$$
\end{itemize}
such that for any 0-cell $A$
\begin{align}\label{mn1}
1_{(\chi^0_{g,f})_A} \circ F_g(\kappa_f)_A (\omega_{g,1,f})_A= 1_{(\chi^0_{g,f})_A} \circ ( \zeta_g)_{F_f(A)},
\end{align}
\begin{align}\label{mn2}\begin{split} &\big( \idn_3\circ (F_g(\omega_{h,f,k})_A) \big)\big(\omega_{g,hf,k}\circ \idn_2 \big)
\big(\idn_{(\chi^0_{ghf,k})_A}\circ (\omega_{g,h,f})_{F_k(A)}\big)=\\
&=\big((\omega_{g,h,fk})_A\circ  \idn_4)\big)
\big(\idn_5\circ (\chi_{g,h})_{ \chi^0_{f,k}}\big)\big((\omega_{gh,f,k})_A\circ \idn_6\big),
\end{split}
\end{align}
for any $g,h,f,k\in G$. Where,
$$\idn_{2}=1_{F_g( \chi^0_{h,f})_{F_k(A)}},\quad  \idn_{3}=1_{(\chi^0_{g,hfk})_A}, \quad \idn_4=1_{F_g F_f( \chi^0_{h,k})_A},$$
$$\idn_{5}=1_{(\chi^0_{gh,fk})_A}, \quad  \idn_{6}= 1_{(\chi^0_{g,h})_{F_fF_k(A)}}.$$
In equation \eqref{mn2}, we are omitting the associativity isomorphisms of the pseudofunctors $F_g$. In the following diagrams we shall denote by $\li{g}$ the pseudofunctor $F_g$,  the composition of functors as juxtaposition and the tensor product of pseudonatural transformations also by juxtaposition. Diagrammatically, we have modifications
\[
\xymatrixrowsep{0.5in}
\xymatrixcolsep{0.5in}
\xymatrix{
\li{g}\ \li{h}\ \li{f} \ar[d]_{1_{\li{g}}\otimes \chi_{h,f}} \ar[rr]^{\chi_{g,h}\otimes 1_{\li{f}}}&& \li{gh}\ \li{f} \ar[d]^{\chi_{gh,f}}\\
\li{g}\ \li{hf} \ar[rr]_{\chi_{g,hf}}&& \li{ghf} {\ar @{} [llu] |{\Downarrow\omega_{g,h,f}}},
}
\]

such that the next diagrams are equal for all $g,h,f,k \in G$,  

\begin{equation}\label{omega-axiom} 
\xymatrixrowsep{0.2in}
\xymatrixcolsep{0.2in}
\xymatrix{
  &  & \li{gh}\ \li{f}\ \li{f}\ar[rrrr]^{\chi_{gh,f}\otimes 1_{\li{k}}} {\ar @{} [ddrr] |{\Downarrow\omega_{g,h,f}\otimes 1_{\li{k}}}} &  &  &  & \li{ghf}\ \li{k}\ar[rrdd]^{\chi_{ghf,k}} &  &  \\ 
  &  &  &  &  &  &  &  &  \\ 
 \li{g}\ \li{h}\ \li{f}\ \li{k}\ar[rruu]^{\chi_{g,h}\otimes 1_{\li{f}}\otimes 1_{\li{k}}}  \ar[rrdd]_{1_{\li{g}}\otimes 1_{\li{h}}\otimes \chi_{f,k}}\ar[rrrr]^{1_{\li{g}}\otimes \chi_{h,f}\otimes 1_{\li{k}}} &  &  &  & \li{g}\ \li{hf}\ \li{k}\ar[rruu]^{\chi_{g,hf}\otimes 1_{\li{k}}}\ar[rrdd]^{1_{\li{g}}\otimes \chi_{hf,k}} &  &  &  & \li{ghfk} {\ar @{} [llll] |{\Downarrow \omega_{g,hf,k}}}\\ 
  &  &  &  &  &  &  &  &  \\ 
  &  &  \li{g} \ \li{h}\ \li{fk} {\ar @{} [rrruu] |{\Downarrow 1_{\li{g}}\otimes \omega_{h,f,k}}} \ar[rrrr]_{1_{\li{g}}\otimes \chi_{h,fk}} &  &  &  & \li{g}\ \li{hfk}\ar[rruu]_{\chi_{g,hfk}} &  & 
}
\end{equation}
\begin{center} \Huge{\rotatebox[origin=c]{-90}{$=$}} 
\end{center}
\[\xymatrixrowsep{0.2in}
\xymatrixcolsep{0.2in}
\xymatrix{
  &  & \li{gh}\ \li{f}\ \li{k} \ar[ddrr]^{1_{\li{gh}}\otimes\chi_{f,k}} \ar[rrrr]^{\chi_{gh,f}\otimes 1_{\li{k}}}  &  &  &  & \li{ghf}\ \li{k}\ar[rrdd]^{\chi_{ghf,k}} {\ar @{} [lldd] |{\Downarrow \omega_{gh,f,k}}} &  &  \\ 
  &  &  &  &  &  &  &  &  \\ 
 \li{g}\ \li{h}\ \li{f}\ \li{k} \ar[rruu]^{\chi_{g,h}\otimes 1_{\li{f}}\otimes 1_{\li{k}}}   \ar[rrdd]_{1_{\li{g}}\otimes 1_{\li{h}}\otimes \chi_{f,k}}   {\ar @{} [rrrr] |{\Downarrow c_{\chi_{g,h},\chi_{f,k}}}} &  &  &  & \li{gh}\ \li{fk}  \ar[rrrr]_{\chi_{gh,fk}} &  &  &  & \li{ghfk} \\ 
  &  &  &  &  &  &  &  &  \\ 
  &  &  \li{g} \ \li{h}\ \li{fk} \ar[uurr]_{\chi_{g,h}\otimes 1_{\li{fk}}}  \ar[rrrr]_{1_{\li{g}}\otimes \chi_{h,fk}} &  &  &  & \li{g}\ \li{hfk}\ar[rruu]_{\chi_{g,hfk}} {\ar @{} [lluu] |{\Downarrow\omega_{g,h,fk}}} &  & 
}\]

We say that a group $G$ acts \emph{trivially} on $\Bc$ if the weak monoidal homomorphism 
$(\Fc,\chi,\omega,\iota,\kappa,\zeta): \dG \to \Pse(\Bc, \Bc)$ is the trivial one. This means that for any $g,h\in G$, the pseudofunctors $F_g$ are the identity, $\chi_{g,h}$ are the identity pseudonatural transformations and all the modifications are identities.

\begin{defi}\label{unital-act} An action $(\Fc,\chi,\omega,\iota,\kappa,\zeta): \dG \to \Pse(\Bc, \Bc)$ is called \emph{unital} if $F_g$ is a unital pseudofunctor,
$F_1=\Id_\Bc$, and  $\chi_{g,1}=\id_{F_g}=\chi_{1,g}$, $\kappa_g=\id=\zeta_g$ for any $g\in G$. 
A unital $G$-action will
be denoted simply by $(\Fc,\chi,\omega)$.
\end{defi}

\begin{defi}\label{strict-act} An action $(\Fc,\chi,\omega,\iota,\kappa,\zeta): \dG \to \Pse(\Bc, \Bc)$ is called \emph{strict} if each pseudofunctor $F_g$ is a 2-functor, and  $F_g\circ F_h=F_{gh}$, and the pseudonatural transformations $\chi_{g,h}$ and the modifications $\omega_{g,h,f}$ are the identities for any $g,h,f\in G$.
\end{defi}

A similar argument as in \cite[Proposition 3.1]{Ga2} applied in this case, allows us to  consider  only unital actions. Assume that $\Bc, \Bc'$ are 2-categories equipped with  unital actions of a group $G$ via 
$$(\Fc,\chi,\omega): \dG \to \Pse(\Bc, \Bc),\,\, (\widetilde \Fc,\widetilde \chi,\widetilde \omega): \dG \to \Pse(\widetilde{\Bc}, \widetilde{\Bc}).$$

\begin{defi} A \emph{$G$-pseudofunctor} 
 between $\Bc$ and $\widetilde{\Bc}$ is a 
triple $(\Hc, \gamma, \Pi)$, where 
\begin{itemize}
 \item $\Hc:\Bc\to \widetilde{\Bc}$ is a unital pseudofunctor,
 
 \item for any $g\in G$, pseudonatural equivalences $\gamma_g:\Hc\circ F_g \to \widetilde F_g\circ \Hc,$
 \item invertible modifications
 \[
\xymatrix{
&\s{f}\HH\li{g} \ar[rr]^{1_{\s{f}}\otimes \gamma_g}&& \s{f}\s{g}\HH \ar[rd]^{\s{\chi_{f,g}}\otimes 1_{\HH}}\\
\HH\li{f}\ \li{g} \ar[ru]^{\gamma_f\otimes 1_{\li{g}}} \ar[rrd]_{1_{\HH}\otimes \chi_{f,g}}&&&& \s{fg}\HH\\
&&\HH\li{fg} \ar[rru]_{\gamma_{fg}} {\ar @{} [uu] |{\Downarrow\Pi_{f,g}}}&&
}
\] 
\end{itemize}
such that such that for all $f,g,h\in G$
\begin{align}\label{G-funct1} \gamma_1=\id_\Hc, \quad \Pi_{g,1}=\idn_{\gamma_g}=\Pi_{1,g},
\end{align}
\begin{equation}\label{Pi-axiom}
\xymatrixrowsep{0.2in}
\xymatrixcolsep{0.2in}
\xymatrix{
  &  &  & \s{f}\s{g}\HH\li{h} \ar[rrd]^{1_{\s{f}}1_{\s{g}}\gamma_h } {\ar @{} [dd] |{\Downarrow1_{\s{f}}\ot\Pi_{g,h}}} &  &  &  \\ 
  & \s{f}\HH\li{g}\li{h}  \ar[rru]^{1_{\s{f}}\gamma_g1_{\li{h}}}\ar[rd]^{1_{\s{f}\HH}\chi_{g,h}} &  &  &  & \s{f}\s{g}\s{h}\HH\ar[ld]_{1_{\s{f}}\s{\chi_{g,h}}1_{\HH}}\ar[rd]^{\s{\chi_{f,g}}1_{\s{h}\HH}} &  \\ 
 \HH\li{f}\li{g}\li{h} {\ar @{} [rr] |{c_{\gamma_f,\chi_{g,h}}}} \ar[ru]^{\gamma_f 1_{\li{g}\ \li{h}}}\ar[rd]_{1_{\HH \li{f}}\chi_{g,h}} &  & \s{f}\HH \li{gh}\ar[rr]^{1_{\s{f}}\gamma_{gh}} &  & \s{f}\s{gh}\HH\ar[rd]_{\s{\chi_{f,gh} 1_{\Hc}}} &  & \s{fg}\s{h}\HH \ar[ld]^{\s{\chi_{fg,h}}1_{\HH}}\\ 
  & \HH\overline{f}\ \overline{gh} \ar[ru]_{\gamma_f 1_{\li{gh}}}\ar[rrd]_{\chi_{f,gh}} &  &  &  & \s{fgh}\HH  {\ar @{} [uu] |{\overset{\widetilde\omega_{f,g,h}}{\Leftarrow}}} &  \\ 
  &  &  & \HH\overline{fgh} \ar[rru]_{\gamma_{fgh}} {\ar @{} [uu] |{\Downarrow\Pi_{f,gh}}}&  &  & 
}
\end{equation}
\begin{center}  \Huge{\rotatebox[origin=c]{-90}{$=$}} 
\end{center}
\[
\xymatrixrowsep{0.2in}
\xymatrixcolsep{0.2in}
\xymatrix{
  &  &  & \s{f}\s{g}\HH\li{h} \ar[ddr]^{\s{\chi_{f,g}}1_{\HH\li{h}}} \ar[rrd]^{1_{\s{f}}1_{\s{g}}\gamma_h } &  &  &  \\ 
  & \s{f}\HH\li{g}\li{h} {\ar @{} [rrr] |{\Downarrow \Pi_{f,g}\ot 1_{\li{h}}}} \ar[rru]^{1_{\s{f}}\gamma_g1_{\li{h}}} &  &  &  & \s{f}\s{g}\s{h}\HH \ar[rd]^{\s{\chi_{f,g}}1_{\s{h}\HH}} {\ar @{} [dl] |{c_{\widetilde\chi_{f,g},\gamma_h}}}&  \\ 
 \HH\li{f}\li{g}\li{h} \ar[rr]^{1_{\HH}\chi_{f,g}1_{\li{h}}} \ar[ru]^{\gamma_f 1_{\li{g}\li{h}}}\ar[rd]_{1_{\HH \li{f}}\chi_{g,h}} &  & \HH \li{fg}\ \li{h} \ar[rr]^{\gamma_{fg}1_{\li{h}}} \ar[ddr]^{1_{\HH}\chi_{fg,h}} &  &  \s{fg}\HH\li{h} {\ar @{} [dd]+<6mm> |{\Downarrow\Pi_{fg,h}}} \ar[rr]^{1_{\s{fg}}\gamma_h}&  & \s{fg}\s{h}\HH \ar[ld]^{\s{\chi_{fg,h}}1_{\HH}}\\ 
  & \HH\overline{f}\ \overline{gh} {\ar @{} [ur] |{\overset{1_{\HH}\ot\omega_{f,g,h}}{\Leftarrow}}} \ar[rrd]_{1_{\HH}\chi_{f,gh}} &  &  &  & \s{fgh}\HH  &  \\ 
  &  &  & \HH\overline{fgh} \ar[rru]_{\gamma_{fgh}} &  &  & 
}
\]
holds in \doscat$(\bb,\bb)$. In the above diagrams, we are using the comparison constraints $c$ defined in \eqref{comparison-const}.
\end{defi}

\begin{defi} Assume that $(\HH,\gamma, \Pi), (\HH',\gamma', \Pi')$ are $G$-pseudofunctors.
A $G$-\textit{pseudonatural} transformation is a pair $(\theta, \{\theta_g\}_{g\in G})$, where $\theta: \HH\to \HH'$ is a pseudonatural transformation,   and  $\theta_g$ are invertible modifications
\[
\xymatrixrowsep{0.6in}
\xymatrixcolsep{0.6in}
\xymatrix{
\HH\li{g} \ar[r]^{\gamma_g } \ar[d]_{\theta\ot \unoli{g} } & \s{g}\HH \ar[d]^{\unos{g}\ot \theta} \\
\HH'\li{g}\  \ar[r]^{\gamma_{g}'} & \s{g}\HH' \ar@{}[ul]|-{\Downarrow \theta_g}
}\]

such that for all $g,f \in G$, the equation
\[
\xymatrixrowsep{0.3in}
\xymatrixcolsep{0.3in}
\xymatrix{
\HH\li{g}\ \li{f} \ar@{}[rrd]|-{\Downarrow \theta_g\unoli{f}} \ar[rr]^{\gamma_g \unoli{f}} \ar[d]_{\theta \unoli{g} \unoli{f}} &&  \s{g}\HH\li{f} \ar@{}[rrd]|-{\Downarrow \unos{g}\theta_f} \ar[d]_{\unos{g}\theta\unoli{f}} \ar[rr]^{\unos{g}\gamma_f} && \s{g}\s{f}\HH \ar@{}[rrd]|-{\Downarrow c_{\theta,\s{\chi_{g,f}}}} \ar[d]_{\unos{g}\unos{f}\theta}  \ar[rr]^{\s\chi_{g,f}1_{\HH}} &&\s{gf}\HH \ar[d]^{\unos{gf}\theta} &\\
\HH'\li{g}\ \li{f} \ar@/_/[rrrd]_{1_{\HH'}\chi_{g,f}}\ar[rr]_{\gamma_{g}'\unoli{f}} && \s{g}\HH'\li{f} \ar[rr]^{\unos{g}\gamma_f'}&& \s{g}\s{f}\HH' \ar[rr]_{\s{\chi_{g,f}}1_{\HH'}}&& \s{gf}\HH' & \\
&&&\HH'\li{gf} \ar@/_/[rrru]_{\gamma_{gf}'} \ar@{}[u]|-{\Downarrow \Pi_{g,f}'}
}\]

\begin{center}  \Huge{\rotatebox[origin=c]{-90}{$=$}} 
\end{center}

\[
\xymatrixrowsep{0.3in}
\xymatrixcolsep{0.3in}
\xymatrix{
\HH\li{g}\ \li{f} \ar[d]_{\theta\unoli{g}\unoli{f}} \ar@/_/[rrrd]_{1_{\HH}\chi_{g,f}}\ar[rr]_{\gamma_{g}\unoli{f}} && \s{g}\HH\li{f} \ar[rr]^{\unos{g}\gamma_f}&& \s{g}\s{f}\HH' \ar[rr]_{\s{\chi_{g,f}}1_{\HH}}&& \s{gf}\HH \ar[d]^{\unos{gf}\theta} \\ \HH'\li{g}\ \li{f} \ar@{}[rrr]<-3mm>|-{\Downarrow c_{\theta,\chi_{g,f}}^{-1}} \ar@/_/[rrrd]_{1_{\HH'}\chi_{g,f}}
&&&\HH\li{gf} \ar[d]_{\theta\unoli{gf}}\ar@/_/[rrru]_{\gamma_{gf}} \ar@{}[u]|-{\Downarrow \Pi_{g,f}}&&&\s{gf}\HH' \ar@{}[lll]<3mm>|-{\Downarrow \theta_{gf}} \\
&&&\HH'\li{gf} \ar@/_/[rrru]_{\gamma_{gf}'}
}\]
holds in \doscat$(\bb,\bb)$.

\end{defi}

\begin{defi} Assume that $(\theta, \{\theta_g\}_{g\in G}), (\sigma, \{\sigma_g\}_{g\in G}):(\Hc, \gamma, \Pi)\to 
(\widetilde\Hc, \widetilde\gamma, \widetilde\Pi)$ are  
$G$-pseudonatural transformations. A \emph{$G$-modification}
$\alpha: (\theta, \{\theta_g\}_{g\in G})\Rightarrow(\sigma, \{\sigma_g\}_{g\in G})$ is a 
modification $\alpha: \theta \Rightarrow \sigma$ such that

\[
\xymatrixrowsep{0.4in}
\xymatrixcolsep{0.4in}
\xymatrix{
\HH\li{g} \ar[rr]^{\gamma_g } \ar@{}[d]|-{\overset{\alpha \otimes \unoli{g}}{\Leftarrow }} \ar@/_3pc/[d]_{\sigma \otimes \unoli{g} } \ar@/^3pc/[d]^{\theta \otimes \unoli{g} }  && \s{g}\HH \ar[d]^{\unos{g}\theta_g} \\
\HH'\li{g}\  \ar[rr]_{\gamma_{g}'} && \s{g}\HH' \ar@{}[ul]|-{\Downarrow \theta_g}
}\]
\begin{center}  \Huge{\rotatebox[origin=c]{-90}{$=$}} 
\end{center}
\[
\xymatrixrowsep{0.4in}
\xymatrixcolsep{0.4in}
\xymatrix{
\HH\li{g} \ar[rr]^{\gamma_g } \ar[d]_{\sigma_g\otimes \unoli{g}} \ar@{}[rrd]<-1mm>_{\Downarrow \sigma_g}  && \s{g}\HH   \ar@{}[d]|-{\overset{\unos{g}\otimes \alpha}{\Leftarrow }} \ar@/_3pc/[d]_{ \unos{g}\otimes \sigma } \ar@/^3pc/[d]^{\unos{g}\otimes \theta } \\
\HH'\li{g}\  \ar[rr]_{\gamma_{g}'} && \s{g}\HH' 
}\]
\end{defi}

Assume that $(\Hc^1, \gamma^1, \Pi^1), (\Hc^2, \gamma^2, \Pi^2), (\Hc^3, \gamma^3, \Pi^3)$ are $G$-pseudofunctors, and $(\theta, \{\theta_g\}_{g\in G}): (\Hc^1, \gamma^1, \Pi^1)\to (\Hc^2, \gamma^2, \Pi^2)$,
$(\sigma, \{\sigma_g\}_{g\in G}): (\Hc^2, \gamma^2, \Pi^2)\to (\Hc^3, \gamma^3, \Pi^3)$ are
$G$-pseudonatural transformations. The composition
$$ (\sigma, \{\sigma_g\}_{g\in G}) \circ  (\theta, \{\theta_g\}_{g\in G})=(\rho, \{\rho_g\}_{g\in G})$$
is defined as follows. The pseudonatural transformation 
$\rho=\sigma\circ \theta$. For any 0-cell $A\in \Bc$ and  any $g\in G$
$$(\rho_g)_A= \big((\sigma_g)_A\circ \id_{  \theta^0_{F_g(A)}} \big)
\big( \id_{\widetilde F_g(\sigma^0_A)}\circ (\theta_g)_A)\big). $$
Here, we are also ommiting the associativity constraints of the pseudofunctor $F_g$. The composition of modifications of $G$-categories is the usual composition of modi\-fications.

\begin{defi} $\Pse^G(\Bc, \widetilde\Bc)$ is the 2-category in which 0-cells are pseudofunctors of $G$-categories, 
1-cells are pseudonatural transformations of $G$-categories and 2-cells are modifications of $G$-categories.\end{defi}

 The next result is a consequence of \cite[Corollary 8.3]{GPS}.
\begin{prop}   $\Pse^G(\Bc, \widetilde\Bc)$ is a 2-category. \qed
\end{prop}

\begin{defi} We say that the 2-categories $\Bc$ and $ \widetilde\Bc$ are $G$-\textit{biequivalent} if there exists a $G$-pseudofunctor $\Hc:\Bc\to \widetilde\Bc$ that is also a biequivalence.
\end{defi}

\begin{lema}[Transport of structure]\label{transport}
Let $\Bc$ be a 2-category with an action of $G$ given by  $(\Fc,\chi,\omega)$. Let  $\HH:\Bc\to \Bc'$ be a biequivalence, \[L_g:\Bc'\to \Bc', \  \  \gamma_g:\HH \circ F_g\to L_g\circ \HH\]   a $G$-indexed family of pseudofunctors and pseudonatural equivalences, respectively. Then, there is a way to endowed $\Bc'$ with a $G$-action   $(L,\chi',\omega')$ such that $(\HH,\gamma, \Pi ):\Bc\to \Bc'$ is a $G$-biequivalence .
\end{lema}
\begin{proof} Since $\gamma_g$ and $\chi_{f,g}$ are psedonatural equivalences, we can simultaneously provide the datum $\Pi_{f,g}$ and the pseudonatural equivalences
$\chi'_{f,g}:L_f\circ L_g\to L_{fg}$, $f, g\in G$. Now, axiom \ref{Pi-axiom} uniquely determines the modifications $\omega_{f,g,h}'$. Axiom \ref{omega-axiom} follows from the corresponding axioms of $G$-action via $(\Fc,\chi,\omega)$. The pseudofunctor $(\HH,\gamma, \Pi ):\Bc\to \Bc'$ is a $G$-biequivalence by construction.
\end{proof}

\begin{cor}\label{action with 2-functors}
Every 2-category with a $G$-action is $G$-biequivalent to a 2-category where $G$ acts by 2-functors, that is, all $F_g$ are 2-functors.
\end{cor}
\begin{proof}
By the coherence of theorem for pseudofunctor, see \cite[Section 2.3]{tricat-nick}, every bicategory $\Bc$ is biequivalent to a 2-category $\operatorname{st}(\Bc)$ such that every pseudo-functor $F:\operatorname{st}(\Bc)\to \operatorname{st}(\Bc)$ is pseudo-natural equivalent to a 2-functor. Then applying Lemma \ref{transport} we can transport the action of $\Bc$ to a $G$-biequivalent action on $\operatorname{st}(\Bc)$   where $G$ acts by 2-functors.
\end{proof}

 \section{Coherence for group actions on 2-categories}\label{Section:coherence}
 
 The main result of this section is to prove the following coherence theorem for a group action on a 2-category.
 
 \begin{teo}[Coherence for group actions on 2-categories]\label{coherence-groupact} Let $G$ be a  group. Every 2-category with an action of $G$ is $G$-biequivalent to a 2-category with a strict action of $G$.\qed
\end{teo}

Assume $\Bc$ is a 2-category equipped with a unital action of $G$, $(\Fc,\chi,\omega): \dG \to \Pse(\Bc, \Bc)$. By Corollary \ref{action with 2-functors} we can assume that  $F_g$ is a 2-functor for any $g\in G$. We shall first construct  a 2-category $\bb[G]$ with a strict action of $G$.

Objects of $\Bc[G]$ are triples $(A,\theta,\alpha)$, where  $A=\{A_g\}_{g}$ is a $G$-indexed family of objects, $\theta=\{\theta_{g,h}:F_g(A_h)\to A_{gh}\}_{g,h\in G}$ is a $G\times G$-indexed family of 1-cell equivalences  and 
\[
\xymatrixrowsep{0.4in}
\xymatrixcolsep{0.4in}
\xymatrix{
F_{g}F_{h}(A_{f}) \ar[d]_{F_g(\theta_{h,f})} \ar[rr]^{(\chi_{g,h}^0)_{A_f}}&& F_{gh}(A_{f}) \ar[d]^{\theta_{gh,f}}\\
F_g(A_{hf}) \ar[rr]_{\theta_{g,hf}}&& A_{ghf} {\ar @{} [llu] |{\Downarrow\alpha_{g,h,f}}},
}
\] a $G\times G\times G$-index family of isomorphism 2-cells, such 
\[\theta_{1,g}=I_{A_g}, \quad \alpha_{1,h,f}=\operatorname{id}, \quad \alpha_{g,1,f}=\operatorname{id} \]
that for all $g,h,f,k$, and  equation 
\begin{equation}\label{obj-axiom} 
\xymatrixrowsep{0.2in}
\xymatrixcolsep{0.2in}
\xymatrix{
  &  & \li{gh}\ \li{f}\ A_k\ar[rrrr]^{\chi_{gh,f}^0} {\ar @{} [ddrr] |{\Downarrow\omega_{g,h,f}}} &  &  &  & \li{ghf}\ A_k\ar[rrdd]^{\theta_{ghf,k}} &  &  \\ 
  &  &  &  &  &  &  &  &  \\ 
 \li{g}\ \li{h}\ \li{f}\ A_{k}\ar[rruu]^{\chi_{g,h}^0\otimes 1_{\li{f}}}  \ar[rrdd]_{1_{\li{g}}\otimes 1_{\li{h}}\otimes \theta_{f,k}}\ar[rrrr]^{1_{\li{g}}\otimes \chi_{h,f}^0} &  &  &  & \li{g}\ \li{hf}\ A_k\ar[rruu]^{\chi_{g,hf}^0}\ar[rrdd]^{1_{\li{g}}\otimes \theta_{hf,k}} &  &  &  & A_{ghfk} {\ar @{} [llll] |{\Downarrow \alpha_{g,hf,k}}}\\ 
  &  &  &  &  &  &  &  &  \\ 
  &  &  \li{g} \ \li{h}\ A_{fk} {\ar @{} [rrruu] |{\Downarrow 1_{\li{g}}\otimes \alpha_{h,f,k}}} \ar[rrrr]_{1_{\li{g}}\otimes \theta_{h,fk}} &  &  &  & \li{g}\ A_{hfk}\ar[rruu]_{\theta_{g,hfk}} &  & 
}
\end{equation}
\begin{center} \Huge{\rotatebox[origin=c]{-90}{$=$}} 
\end{center}
\[\xymatrixrowsep{0.2in}
\xymatrixcolsep{0.2in}
\xymatrix{
  &  & \li{gh}\ \li{f}\ A_k \ar[ddrr]^{1_{\li{gh}}\otimes\theta_{f,k}} \ar[rrrr]^{\chi_{gh,f}^0}  &  &  &  & \li{ghf}\ A_k\ar[rrdd]^{\theta_{ghf,k}} {\ar @{} [lldd] |{\Downarrow \alpha_{gh,f,k}}} &  &  \\ 
  &  &  &  &  &  &  &  &  \\ 
 \li{g}\ \li{h}\ \li{f}\ A_k \ar[rruu]^{\chi_{g,h}^0\otimes 1_{\li{f}}}   \ar[rrdd]_{1_{\li{g}}\otimes 1_{\li{h}}\otimes \theta_{f,k}}   {\ar @{} [rrrr] |{\Downarrow (\chi_{g,h})_{\theta_{f,k}}}} &  &  &  & \li{gh}\ A_{fk}  \ar[rrrr]_{\theta_{gh,fk}} &  &  &  & A_{ghfk} \\ 
  &  &  &  &  &  &  &  &  \\ 
  &  &  \li{g} \ \li{h}\ A_{fk} \ar[uurr]_{\chi_{g,h}^0}  \ar[rrrr]_{1_{\li{g}}\otimes \theta_{h,fk}} &  &  &  & \li{g}\ A_{hfk}\ar[rruu]_{\theta_{g,hfk}} {\ar @{} [lluu] |{\Downarrow\alpha_{g,h,fk}}} &  & 
}\] holds in $\bb(F_g (F_h(F_f(A_k)), A_{ghfk})$. If $(A,\theta, \alpha)$ is a 0-cell, the identity 1-cell $I_{(A,\theta, \alpha)}$ is defined as follows. $I_{(A,\theta, \alpha)}=(I_{A_g}, l)$, where 
$l_{g,h}=\id_{\theta_{g,h}}$, for any $g,h\in G$.

If $(A,\theta, \alpha)$ and $(B,\rho ,\beta)$ are objects in $\Bc[G]$, a 1-cell   is a  pair $(X,l )$, where  $X=\{X_g:A_g\to B_g\}$ is a $G$-indexed family of 1-cells and 
\[
\xymatrixrowsep{0.4in}
\xymatrixcolsep{0.4in}
\xymatrix{
F_{g}(A_{h}) \ar[d]_{\theta_{g,h}} \ar[rr]^{F_g(X_h)}&& F_g(B_h) \ar[d]^{\rho_{g,h}}\\
A_{gh} \ar[rr]_{X_{gh}}&& B_{gh} {\ar @{} [llu] |{\Downarrow l_{g,h}}},
}
\] is a $G\times G$-indexed family of isomorphism 2-cells, such that for all $f,g,h\in G$, $l_{1,g}=\operatorname{id}_{X_g}$ and  equation
\begin{equation}\label{axiom-1-cell B[G]}
\xymatrixrowsep{0.4in}
\xymatrixcolsep{0.4in}
\xymatrix{
\li{f}\ \li{g}(A_{h}) \ar[d]_{\li{f}(\theta_{g,h})} \ar[rr]^{\li{f}\ \li{g} (X_h)}&& \li{f}\ \li{g} (B_h) \ar[d]^{\li{f}\rho_{g,h})} \ar[rr]^{\chi_{f,g}^0} && \li{fg}(B_h) \ar[d]^{\rho_{fg,h}}\\
\li{f}(A_{gh})  \ar[rrd]_{\theta_{f,gh}} \ar[rr]_{\li{f}(X_{gh})}&& \li{f}(B_{gh}) {\ar @{} [d] |{\Downarrow l_{f,gh}}} {\ar @{} [llu] |{\Downarrow \li{f}(l_{g,h})}} \ar[rr]_{\rho_{f,gh}} && {\ar @{} [llu] |{\Downarrow \beta_{f,g,h}}} B_{fgh}\\
&& A_{fgh} \ar[rru]_{X_{fgh}} &&
}
\end{equation}
\begin{center}  \Huge{\rotatebox[origin=c]{-90}{$=$}} 
\end{center}
\[
\xymatrixrowsep{0.4in}
\xymatrixcolsep{0.4in}
\xymatrix{
\li{f}\ \li{g}(A_{h}) \ar[d]_{\li{f}(\theta_{g,h})} \ar[rrd]_{\chi^0_{f,g}} \ar[rr]^{\li{f}\ \li{g} (X_h)}&& \li{f}\ \li{g} (B_h) {\ar @{} [d] |{\Downarrow (\chi_{f,g})_{X_h}}} \ar[rr]^{\chi_{f,g}^0} && \li{fg}(B_h) \ar[d]^{\rho_{fg,h}}\\
\li{f}(A_{gh})  \ar[rrd]_{\theta_{f,gh}} {\ar @{} [rr] |{\Downarrow \alpha_{f,g,h}}} && \li{fg}(A_{h}) \ar[rru]_{\li{fg}(X_h)} \ar[d]_{\theta_{fg,h}} {\ar @{} [rr] |{\Downarrow l_{fg,h}}} && \\
&& A_{fgh} \ar[rru]_{X_{fgh}} &&
}
\] holds in $\bb(F_f(F_g (A_h)),B_{fgh})$. If $(X,l),$ $(Y,s)$ are 1-cells, a 2-cell  $m:(X,l)\Rightarrow (Y,s)$ is a $G$-indexed family of 2-cells $m=\{m_g:X_g\to Y_g\}$ such that for all $g,f \in G$,  equation

\begin{equation}\label{axiom-2-cell B[G]}
\xymatrixrowsep{0.4in}
\xymatrixcolsep{0.5in}
\xymatrix{
F_{g}(A_{h}) \ar[d]_{\theta_{g,h}} \ar@/^2pc/[rr]^{F_g(X_h)}&& F_g(B_h) \ar[d]^{\rho_{g,h}} {\ar @{} [ll] |{\Downarrow l_{g,h}}} \\
A_{gh} \ar@/^2pc/[rr]_{X_{gh}} \ar@/_2pc/[rr]_{Y_{gh}} && B_{gh} {\ar @{} [ll] |{\Downarrow m_{gh}}}
}
\end{equation}
 \begin{center}  \Huge{\rotatebox[origin=c]{-90}{$=$}} 
\end{center}

\[
\xymatrixrowsep{0.4in}
\xymatrixcolsep{0.5in}
\xymatrix{
F_{g}(A_{h})  {\ar @{} [rr] |{\Downarrow F_g(m_h)}}\ar[d]_{\theta_{g,h}} \ar@/^2pc/[rr]^{F_g(X_h)} \ar@/_2pc/[rr]^{F_g(Y_h)} && F_g(B_h) \ar[d]^{\rho_{g,h}}\\
A_{gh}  \ar@/_2pc/[rr]_{Y_{gh}} && {\ar @{} [ll] |{\Downarrow s_{g,h}}} B_{gh} 
}
\] 
 holds in $\bb(F_g(A_h),B_{gh})$. 
 
 The (vertical) composition in each category $\Bc[G]((A,\theta, \alpha),(B,\rho ,\beta))$ is defined pointwise.

 Now, let us define the horizontal composition 
 $\circ:\Bc[G]((A,\theta, \alpha),(B,\rho ,\beta))\times \Bc[G]((C,\kappa,\gamma),(A,\theta, \alpha))\to \Bc[G]((C,\kappa,\gamma),(B,\rho ,\beta)).$
 If $(A,\theta, \alpha)$ and $(B,\rho ,\beta)$ are 0-cells, and $$(X,l) \in \Bc[G]((A,\theta, \alpha),(B,\rho ,\beta)),\,\, (Y,s)\in \Bc[G]((C,\kappa,\gamma),(A,\theta, \alpha))$$ are 1-cells, define 
 $$(X,l)\circ (Y,s)=(Z,t), $$
 where $Z_g=X_g\circ Y_g$, and $t_{g,h}= \big( 1_{X_{gh}}\circ s_{g,h} \big)\big(  l_{g,h}\circ 1_{F_g(Y_h)}\big)$, for any $g,h\in G$.
 The horizontal composition of 2-cells in $\bb[G]$ is just the horizontal composition of 2-cells in $\bb$.
\begin{lema} $\bb[G]$ is a 2-category endowed with a strict action of $G$.
\end{lema}
\pf The proof that   $\bb[G]$ is indeed a 2-category follows by a straightforward calculation. Let us define now a canonical strict action of $G$ on the 2-category $\bb[G]$. For any $g\in G$ define the 2-functors $L_g:\Bc[G]\to \Bc[G]$ as follows. If $(A,\theta,\alpha)$ is a 0-cell,  $g,x\in G$, then
\[L_g(A)_x=A_{xg}, \quad Lg(\theta)_{x,y}=\theta_{x,yg}, \quad L_g(\alpha)_{x,y,z}=\alpha_{x,y,zg}.\] If $(X,l):(A,\theta,\alpha) \to (B,\rho,\beta)$ is a 1-cell, \[L_g(X)_x=X_{xg}, \quad L_g(l)_{x,y}=l_{xyg}.\] If $m:(X,l)\Rightarrow (Y,s)$ is a 2-cell, then $L_g(m)_x=m_{xg}$, for any $x\in G$.
Since  the $L_g $ are 2-functors such that $L_g\circ L_h=L_{gh}$ for all $g, h\in G$ and $L_e=\operatorname{Id}_{\bb[G]}$, $L$ defines a strict action of $G$  on $\bb[G]$.
\epf

There is a pseudofunctor $\Hc:\Bc\to \Bc[G]$ defined as follows.
 If $A$  is a 0-cell in $\Bc$,  then \[\HH(A)= (\{F_g(A)\}, (\chi^0_{g,h})_A,\omega_{g,h,f})_{f,g,h\in G},\]if $X:A\to B$ is a 1-cell, then 
 $\HH(X)= (F_g(X),(\chi_{g,h})_X)$ and for 2-cells $m:X\to Y$, 
 $\HH(m)_g=F_g(m)$, where $f,g,h\in G$. The fact that $\omega$ are modifications implies that $\HH(X)$ is indeed a 1-cell in $\Bc[G]$.
The following proposition implies immediately Theorem \ref{coherence-groupact}
\begin{prop}\label{propiedades de HH}  $\Hc:\Bc\to \Bc[G]$ is a  $G$-biequivalence.\end{prop}
\pf 
If $(A,\theta,\alpha)$ is an object in $\bb[G]$, then  the 1-equivalences $\theta_{g,e}:\HH(A_e)_g\to A_g$ and the 2-cells
\[
\xymatrixrowsep{0.4in}
\xymatrixcolsep{0.4in}
\xymatrix{
F_{g}\HH(A_e)_h \ar[d]_{F_g(\theta_{h,e})} \ar[rr]^{\chi_{g,h}^0}&& \HH(A_e)_{gh} \ar[d]^{\theta_{gh,e}}\\
F_g(A_{h}) \ar[rr]_{\theta_{g,h}^0}&& A_{ghf} {\ar @{} [llu] |{\Downarrow\alpha_{g,h,e}}},
}
\]defines a 1-equivalence from  $\HH(A_1)$ to $A$, that is, $\HH$ is bi-essentially
surjective.


Let $A$ and $B$ be objects in $\Bc$, and $(X,l):\HH(A)\to \HH(B)$ be a 1-cell in $\bb[G]$. The invertible 2-cells $l_{g,1}:\HH(X_1)_g\to X_g$ define an invertible 2-cell from $\HH(X_1)$ to $X$. Then $\HH$ is locally essentially surjective.

If $X,Y, \in \Bc(A,B)$ and  $f,f':X\to Y$ such that $\Hc(f)=\Hc(f')$. Thus, $\Hc(f)_1=\Hc(f')_1$, 
but since we are considering a unital action,  $f=\Hc(f)_1=\Hc(f')_1=f'$, that is, $\HH$ is locally faithful. Suppose $w:\HH(X)\to \HH(Y)$ is a 2-cell in 
$\bb[G]$, condition \eqref{axiom-2-cell B[G]} implies that $w_g=F_g(m_1)$, then $w=\HH(w_1)$. Since, $\HH$ is bi-essentially surjective and locally  fully faithful, $\HH$ is a biequivalence.

To see that $\Hc$ has a canonical structure of $G$-pseudofunctor, we note that \[(\Hc\circ F_g)_x=F_x\circ F_g,\  \  \  \   \  \  (L_g\circ \Hc)_x=F_{xg},\] for any $x,g\in G$. Then, using the  pseudonatural transformations $\chi_{x,g}:F_{x}\circ F_g\to F_{xg}$, we define a pseusonatural transformation
\[\gamma_g:\Hc\circ F_g\to L_g\circ \Hc,\]
as follows. For any 0-cell $A\in \Obj(\Bc)$ we have to define an equivalence 1-cell $\gamma^0_A:\Hc\circ F_g(A)\to L_g\circ \Hc(A)$ in $\bb[G]$. Set $\gamma^0_A=(X,l)$, where, for any $x,f,h\in G$ 
$$X_x=(\chi^0_{x,g})_A, \quad l_{f,h}=(\omega^{-1}_{f,h,g})_A.$$
Axiom \eqref{omega-axiom} implies that morphisms $l_{f,h}$ fulfill condition \eqref{axiom-1-cell B[G]}. Thus, $\gamma^0_A$ is indeed a 1-cell in $\bb[G]$. To complete the definition of of the pseudonatural equivalence $ \gamma_g $, we have to define, 2-cells in $\bb[G]$
$$( \gamma_g )_X:  \gamma^0_B\circ 
\Hc F_g(X)\to L_g\Hc(X)\circ \gamma^0_A,$$
for any 1-cell $X\in\Bc(A,B)$. Set $\big( ( \gamma_g )_X \big)_x=(\chi_{x,g})_X$, for any $x\in G$. The fact that $\omega$ are modifications, imply that 2-cells  $\big( ( \gamma_g )_X \big)_x$ satisfy \eqref{axiom-2-cell B[G]}.
To define the modifications 

 \[
\xymatrix{
&L_{f}\HH F_{g} \ar[rr]^{1_{L_{f}}\otimes \gamma_g}&& L_{f}L_{g}\HH \ar[rd]^{id }\\
\HH F_{f}\ F_{g} \ar[ru]^{\gamma_f\otimes 1_{F_{g}}} \ar[rrd]_{1_{\HH}\otimes \chi_{f,g}}&&&& L_{fg}\HH\\
&&\HH F_{fg} \ar[rru]_{\gamma_{fg}} {\ar @{} [uu] |{\Downarrow\Pi_{f,g}}}&&
}
\] we note that
\[[(1_{L_f\otimes \gamma_g})\circ (\gamma_f\otimes 1_{F_g})]_x= \chi_{xf,g} \circ  (\chi_{x,f}\otimes 1_{F_g}), \  \  x,f,g \in G,\]
and
\[[(1_{\HH}\otimes \chi_{f,g})\circ (\gamma_{fg})]_x=\chi_{x,fg}\circ (1_{F_x}\otimes \chi_{f,g}), \  \  x,f,g \in G.\] 
Then we define $(\Pi_{f,g})_x=\omega_{x,f,g}$ for all $x,g,f\in G$.

Since $\omega_{x,f,g}$ are modifications,  $\Pi_{g,h}$ turns out to be modifications for any $g,h\in G$. Condition described in diagram \eqref{Pi-axiom} is exactly  diagram \eqref{omega-axiom}.  
\epf

\section{The equivariant 2-category}\label{Section:equivariant}
Let $G$ be a group. Denote by $\mathcal{I}$ the unit 2-category endowed with the trivial action  of $G$, and
assume that $\Bc$ is a 2-category with an action of $G$.
\begin{defi}\label{def-equiv} The \emph{equivariant 2-category } is $\Bc^G= \Pse^G(\mathcal{I},\Bc)$. 0-cells, 1-cells and 2-cells in $\Bc^G$ will be called \textit{equivariant} 0-cells, 1-cells and 2-cells, respectively.
\end{defi}

\begin{prop} Assume  $\Bc$ and $ \widetilde\Bc$ are $G$-biequivalent. Then the 2-categories $\Bc^G$, $\widetilde\Bc^G$ are biequivalent.
\end{prop}
\pf Straightforward.\epf

\begin{lema}\label{forget-2funct} There exists a forgetfull 2-functor $\Phi:\Bc^G\to \Bc$.
 \qed
\end{lema}
\pf If $(\Hc, \Pi,\gamma)$ is an equivariant 0-cell in $\Bc^G$ , then $\Phi(\Hc, \Pi,\gamma)=\Hc(\star)$. If $(\theta, \{\theta_g\}_{g\in G})$ is an equivariant 1-cell, then $\Phi(\theta, \{\theta_g\}_{g\in G})=\theta$. On 2-cells the functor $\Phi$ is the identity.
\epf

\subsection{Unpacking definition of equivariantization}\label{section:unpacking} We shall explicitly describe the 2-category $ \Bc^G $. This would allows us to show concrete examples and obtain some results in Section \ref{Section-center}.

\medbreak

We shall assume that there is a unital action of $G$  on the 2-category $\Bc$ such that all pseudofunctors $F_g$ are 2-functors. This is possible using Corollary \ref{action with 2-functors}. The 2-category $\Bc^G$ has 0-cells
triples $(A, \{U_g\}_{g\in G}, \{\Pi_{g, h}\}_{g, h\in G})$, where
\begin{itemize}
 \item[$\bullet$] $A$ is a 0-cell in $\Bc$;
 
 \item[$\bullet$] $U_g$ are invertible 1-cells in $\Bc(A, F_g(A))$;
 
 \item[$\bullet$] $\Pi_{g,h}:(\chi^0_{g,h})_A\circ F_g(U_h)\circ U_g\to U_{gh} $ are isomorphisms 2-cells 
 in the category $\Bc(A,F_{gh}(A))$ such that
\end{itemize}
$$U_1=I_A, \; \Pi_{g,1}=\id_{U_g}= \Pi_{1,g},$$
\begin{align}\label{a-pi-eq}\begin{split}
&\Pi_{f,gh}\big(\id_{(\chi^0_{f,gh})_A}\circ F_f(\Pi_{g,h})\circ \id_{U_f} \big) \big((\omega_{f,g,h})_A\circ\id_{F_fF_g(U_h)F_f(U_g)U_f} \big)= \\  &=  \Pi_{fg,h}\big(\id_{(\chi^0_{fg,h})_AF_{fg}(U_h)}\circ \Pi_{f,g}\big)\big(\id_{(\chi^0_{fg,h})_A}\circ (\chi_{f,g})_{U_h}\circ \id_{F_f(U_g)U_f}\big)
                            \end{split}
\end{align}
for all $g,h,f\in G$. For short, the collection $(A, \{U_g\}_{g\in G}, \{\Pi_{g, h}\}_{g, h\in G})$ will be    denoted  simply as
$(A, U, \Pi)$.

\medbreak

Given two equivariant 0-cells $(A, U, \Pi)$, $(\widetilde A, \widetilde U, \widetilde\Pi)$,
an \textit{equivariant 1-cell } is a pair $(\theta, \{\theta_g\}_{g\in G})\in \Bc^G((A, U, \Pi),(\widetilde A, \widetilde U, \widetilde\Pi))$ where 
\begin{itemize}
\item $\theta:\Bc(A, \widetilde A)$ is a 1-cell,
\item  and for any $g\in G$, $\theta_g:F_g(\theta)\circ U_g \Rightarrow \widetilde U_g\circ \theta$, are invertible 2-cells such that $\theta_1=\id_\theta,$ and such that for any $g,f\in G$
\end{itemize}
\begin{align}\label{equiva-1cell}
\begin{split}
\big(\widetilde \Pi_{g,f}\circ \id_{\theta}\big) \big( \id_{(\chi^0_{g,f})_AF_g(\widetilde U_f)} \circ\theta_g\big)\big(\id_{(\chi^0_{g,f})_A} \circ F_g(\theta_f)\circ\id_{U_g} \big)=\\
=\theta_{gf} \big(\id_{F_{gf}(\theta)}\circ\Pi_{g,f}\big)\big(  (\chi_{g,f})_\theta\circ\id_{F_g(U_f)U_g}\big).
\end{split}
\end{align}
If $(\theta, \{\theta_g\}_{g\in G}),  (\sigma, \{\sigma_g\}_{g\in G}):(A, U, \mu)\to(\widetilde A, \widetilde U, \widetilde\mu)$ are equivariant 1-cells, an \textit{equivariant 2-cell} $\alpha: (\theta, \{\theta_g\}_{g\in G})\to  (\sigma, \{\sigma_g\}_{g\in G})$ is a 2-cell $\alpha:\theta\to \sigma$
such that for all $g\in G$
\begin{align}\label{equiva-2cell} 
(\id_{\widetilde U_{g}}\circ \alpha) \theta_g=
\sigma_g (F_g(\alpha)\circ \id_{U_g}).
\end{align}

Suppose  that  $(A, U, \mu)$, $(\widetilde A, \widetilde U, \widetilde\mu)$, $(A', U', \mu')$
are equivariant 0-cells, and 
$$(\theta, \theta_g):(A', U', \mu')\to(\widetilde A, \widetilde U, \widetilde\mu),
(\sigma, \sigma_g):(A, U, \mu)\to (A', U', \mu')$$
are equivariant 1-cells, then the composition
$(\theta, \theta_g)\circ (\sigma, \sigma_g): (A, U, \mu)\to (\widetilde A, \widetilde U, \widetilde\mu)$ is defined as $(\theta, \theta_g)\circ (\sigma, \sigma_g)=(\theta \circ \sigma, (\theta \circ \sigma)_g)$, where for any $g\in G$
\begin{equation}\label{comp-equiva-1cell}
(\theta \circ \sigma)_g=(\theta_g\circ \id_{\sigma})(\id_{F_g(\theta)}\circ \sigma_g).
\end{equation}

\section{Group actions  from graded tensor categories}\label{Section:examples-from-cat}

Starting with a $G$-graded tensor category $\oplus_{g\in G} \ca_g$, we shall construct a $G$-action on the 2-category of $\ca_1$-representations.

\subsection{Group actions on tensor categories}\label{subsect:groupact} Let $G$ be a finite group and 
$\ca$ be a finite tensor category.
An action of $G$ on $\ca$ consists of the following data:
\begin{itemize}
 \item tensor autoequivalences $(g_*, \xi^g):\ca\to \ca$ for any $g\in G$, 
 \item a natural isomorphism $\zeta:\Id_\ca\to (1)_*$,
 \item and monoidal natural isomorphisms 
$\nu_{g,h}:g_*\circ h_*\to (gh)_*$,
\end{itemize}
 such that for all $X \in \ca$, $g,h,f\in G$
 \begin{align}\label{group-act-tc1}
 (\nu_{gh,f})_X (\nu_{g,h})_{f_*(X)}=(\nu_{g,hf})_X  g_*((\nu_{h,f})_X ), 
 \end{align}
  \begin{align}\label{group-act-tc2} 
 (\nu_{g,1})_X g_*(\zeta_X)=\id_X=  (\nu_{1,g})_X \zeta_{g_*(X)}.
\end{align}
For simplicity, we shall assumed that $(1)_*=\Id_\ca$, $\zeta=\id$ and $\mu_{g,1}=\id=\nu_{1,g}$ 
for all $g\in G$.

If a finite group $G$ acts on a finite tensor category $\ca$, there is associated a new finite tensor category $\ca^G$ 
called the \emph{equivariantization} of $\ca$ by $G$. 
An  object in $\ca^G$  is a pair $(X, s)$, where $X\in \ca$ is an object together
 with isomorphisms $s_g:g_*(X)\to X$ satisfying
\begin{equation}\label{group-act-tc3}
s_1=\id_X,\quad
s_{gh}\circ (\nu_{g,h})_X=s_g\circ g_*(s_h),
\end{equation}
for all $g, h \in G$. A
$G$-\emph{equivariant morphism} $f: (V, s) \to (W, t)$ between $G$-equivariant objects $(V, s)$ and $(W, t)$, 
is a morphism $f: V \to W$ in
$\ca$ such that $f\circ s_g = t_g\circ g_*(f)$ for all $g \in G$. The category $\ca^G$ has a monoidal product as follows. If $(V,s), (W,t)\in \ca^G$, then $(V,s)\ot (W,t)=(V\ot W, r)$, where for any $g\in G$
$$r_g=(s_g\ot t_g) (\xi^g_{V,W})^{-1}.$$

For more details we refer the reader to \cite{BN}, \cite{BuNa}, 
\cite{ENO2}.

\medbreak

There is also associated the graded tensor category $\ca[G]$, with underlying abelian category 
$\ca[G]= \oplus_{g\in G} \ca_g$, where $\ca_g=\ca$ for any $g\in G$. If $X\in \ca$ is an object, the 
object in $\ca_g$ is denoted by $[X,g]$. The tensor product is 
$$[X,g] \ot [Y,h]= [X\ot g_*(Y), gh], \quad X, Y\in \ca, g,h\in G.$$
The reader is refered to \cite{Ta} for the complete monoidal structure of this tensor category.

\subsection{Representations of tensor categories}\label{Section:represent} A left  $\ca$-\emph{module category} over a tensor
category $\ca$ is a   finite   $\ku$-linear  abelian category $\Mo$ equipped with 
\begin{itemize}
 \item[$\bullet$]
a $\ku$-bilinear bi-exact 
bifunctor $\otb: \ca \times \Mo \to \Mo$;
 \item[$\bullet$] natural associativity
and unit isomorphisms $m_{X,Y,M}: (X\otimes Y)\otb M \to X \otb
(Y\otb M)$, $\ell_M: \uno \otb M\to M$, such that 
\begin{equation}\label{left-modulecat1} m_{X, Y, Z\otb M}\; m_{X\otimes Y, Z,
M}= (\id_{X}\otb m_{Y,Z, M})\;  m_{X, Y\otimes Z, M}(a_{X,
Y, Z}\otb \id_{M}),
\end{equation}
\begin{equation}\label{left-modulecat2} (\id_{X}\otb l_M)m_{X,{\bf
1} ,M}= \id_{X \otb M}.
\end{equation}
 \end{itemize}

A \emph{module functor} between module categories $\Mo$ and $\No$ over a
tensor category $\ca$ is a pair $(F,c)$, where
\begin{enumerate}
 \item[$\bullet$] $F:\Mo \to
\No$ is a  left exact  functor;

\item[$\bullet$]   natural isomorphism: $c_{X,M}: F(X\otb M)\to
X\otb F(M)$, $X\in  \ca$, $M\in \Mo$,  such that
for any $X, Y\in
\ca$, $M\in \Mo$:
\begin{align}\label{modfunctor1}
(\id_X \otb  c_{Y,M})c_{X,Y\otb M}F(m_{X,Y,M}) &=
m_{X,Y,F(M)}\, c_{X\otimes Y,M}
\\\label{modfunctor2}
\ell_{F(M)} \,c_{\uno ,M} &=F(\ell_{M}).
\end{align}
\end{enumerate}

Let $\Mo$ and $\No$ be $\ca$-module categories.
We denote by $\Fun_{\ca}(\Mo, \No)$ the category whose
objects are module functors $(F, c)$ from $\Mo$ to $\No$. A
morphism between  $(F,c)$ and $(G,d)\in\Fun_{\ca}(\Mo,
\No)$ is a natural transformation $\alpha: F \to G$ such
that for any $X\in \ca$, $M\in \Mo$:
\begin{gather}
\label{modfunctor3} d_{X,M}\alpha_{X\otb M} =
(\id_{X}\otb \alpha_{M})c_{X,M}.
\end{gather}
We shall also say that $\alpha: F \to G$ is a $\ca$-\emph{module
transformation}.

\medbreak

 Let  $(F, \xi,\phi):\ca\to   \ca$ be a tensor functor and let $(\Mo,
\otb, m)$ be a $ \ca$-module
category. We shall denote by $\Mo^F$ the  $\ca$-module category
with the same underlying
abelian category $\Mo$   and action, associativity and unit  morphisms  defined,
 respectively,  by
\begin{gather}\label{twisted-modc}
  X\otb^F M=F(X)\otb M,\\ m_{X,Y,M}^F=m_{F(X),F(Y),M} (\xi^{-1}_{X,Y}\otb\,
\id_M), \quad l^F_M= l_M (\phi\otb \id_M),\notag
\end{gather}
for all $X, Y\in \ca$, $M\in \Mo$.
Right $\ca$-module and $\ca$-bimodule categories  are defined in a similar way. For the complete definition see \cite{Gr}.

\medbreak

A $\ca$-module category $\Mo$ is \emph{exact} \cite{eo} if, for any projective object
$P\in \ca$, the object $P\otb M$ is projective in $\Mo$ for all
$M\in\Mo$. If $\Mo$ is a left $\ca$-module then $\Mo^{\opp}$ is the  right $\ca$-module over the opposite Abelian category with action 
\begin{equation}\label{opposit-modcat} \Mo^{\opp}\times \ca\to \Mo^{\opp}, (M, X)\mapsto X^*\otb M,
\end{equation}
associativity isomorphisms $m^{\opp}_{M,X,Y}=m_{Y^*, X^*, M}$ for all $X, Y\in \ca, M\in \Mo$. Analogously, if $\Mo$ is a right $\ca$-module category, then $\Mo^{\opp}$ is a left $\ca$-module category. If $\Mo$ is a $\ca$-bimodule category, we denote $ \overline{\Mo}$ the opposite Abelian category, with left and right $\ca$-module structure given as in \eqref{opposit-modcat}.

\subsection{2-categories of representations of tensor categories}

  Suppouse that $\ca$ is a tensor category. The 2-category $\camod$ has as 0-cells, left $\ca$-module categories, 
if $\Mo, \No$ are $\ca$-module categories, then the category $\camod(\Mo, \No)=\Fun_\ca(\Mo, \No).$
Analogously we define the 2-category $\Mod_{\ca}$ of right $\ca$-module categories.

  If $\ca$ is a finite tensor category, the 2-category $\camod_e$ of exact left $\ca$-module categories
is defined in a similar way as $\camod$, with 0-cells being exact left 
$\ca$-module categories. It is known that $\camod_e$ is 2-equivalent 
to ${}_\Do\Mod_e$ if and only if $\ca$ is Morita equivalent to $\Do$. See for example \cite[Thm. 3.4]{FMM}.

\subsection{$G$-Graded tensor categories}\label{graded-tc}
Let $G$ be a finite group. 
A (faithful) $G$-grading on a  finite tensor category $\Do$ is a decomposition 
$\Do=\oplus_{g\in G} \ca_g$, where $\ca_g$ are full abelian subcategories of $\Do$ such that
\begin{itemize}
\item[$\bullet$]  $\ca_g\neq 0$;
 \item[$\bullet$] $\ot:\ca_g\times \ca_h\to \ca_{gh}$ for all
$g, h\in G.$ 
\end{itemize}
 In this case $\ca=\ca_1$ is a tensor subcategory of $\Do$ and each $\ca_g$ is an exact $\ca$-bimodule category.
 We shall assume that 
 $\ca_g\neq 0$ for any $g\in G$. The tensor category $\Do$ is called a $G$-\emph{graded extension} 
 of $\ca$. This class of extensions of tensor categories were studied and classified in \cite{ENO3}.

 If $\Mo$ is a left $\ca$-module category, $X\in \ca_g$, $M\in \Mo$, the functor $G_{X,M}: \overline{\ca_g}\to \Mo$ defined by
 $$G_{X,M}(Y)=({}^*Y\ot X)\otb M,$$
 for any $Y\in \ca_g$, is a $\ca$-module functor. Moreover, the functor $$\Phi:\ca_g\boxtimes_{\ca} \Mo \to \Fun_{\ca}(\overline{\ca_g}, \Mo), \quad \Phi(X\boxtimes M)=G_{X,M},$$
 is an equivalence of $\ca$-module categories.This is a particular case of \cite[Thm. 3.20]{Gr}.

\subsection{The relative center of a bimodule category} The next definition appeared in \cite{GNN}.

\begin{defi}
Let $\ca$ be a tensor category and $\Mo$ a $\ca$-bimodule category. The \emph{
relative center} of $\Mo$ is the category $\Zc_\ca(\Mo)$ of $\ca$-bimodule functors from $\ca$ to $\Mo$.

Explicitly, objects of  $\Zc_\ca(\Mo)$ are pairs $(M,\gamma)$, where $M$ is an objects of $\Mo$ and
$$\gamma=\{\gamma_X:X\overline{\otimes} M\xrightarrow{\sim}M\overline{\otimes}X\}_{X\in \ca}$$ is a natural family of isomorphisms such that
\begin{equation}\label{bien}
\gamma_X\circ \alpha^{-1}_{X,M,Y}\circ \gamma_Y=\alpha^{-1}_{M,X,Y}\circ \gamma_{X\otimes Y}\circ \alpha^{-1}_{X,Y,M},
\end{equation}

where $\alpha_{X,M,Y}:(X\overline{\otimes} M)\overline{\otimes} Y \xrightarrow{\sim}X\overline{\otimes} (M\overline{\otimes} Y)$ are the associativity constraints in $\Mo$.
\end{defi}

Let $\Do=\oplus_{g\in G}\ca_g$ be a $G$-graded tensor category, with $\ca=\ca_1.$ The inclusion functor $\ca\hookrightarrow \Do$ induces the forgetful pseudofunctor $\Hc: {}_\Do\Mod\to  {}_{\ca}\Mod$. 
\begin{prop}\label{cent-forget-relative-center}
 There is a monoidal equivalence
$\Zc(\Hc)\backsimeq \Zc_\ca(\Do).$
\end{prop}

\begin{proof} Let us define the functor  $\Fc:\Zc_\ca(\Do)\to \Zc(\Hc),$ as follows. For any $(V,\gamma)\in  \Zc_\ca(\Do)$ set
$\Fc(V,\gamma)=(W^V,\tau).$ Here, for each $\Mo\in {}_\Do\Mod$, $W^V_\Mo:\Mo\to \Mo$ is the $\ca$-module functor given by 
$$W^V_\Mo(M)=V\overline{\otimes} M.$$
The isomorphisms endowing the functor $W^V_\Mo$ structure of $\ca$-module functor are
$$c_{X,M}:W^V_\Mo(X\overline{\otimes}M)\to X\overline{\otimes}W^V_\Mo(M),$$ given by the following composition:
$$W^V_\Mo(X\overline{\otimes}M)=V\overline{\otimes}(X\overline{\otimes} M)\xrightarrow{m^{-1}_{V,X,M}} (V\otimes X)\overline{\otimes}M\xrightarrow{\gamma^{-1}_X\overline{\otimes}\id_M}(X\otimes V)\overline{\otimes}M$$
$$\xrightarrow{m_{X,V,M}} X\overline{\otimes}(V\overline{\otimes}M)=X\overline{\otimes}W^V_\Mo(M),$$
for any $X\in\ca, M\in \Mo$. It follows  that $(W^V_\Mo,c)$ is a $\ca$-module functor.

Now, we shall explain the definition of $\tau$. Take $\Mo,\No \in {}_{\Do}\Mod$, and $(G,d):\Mo\to \No$ a $\Do$-module functor. Define 
$$\tau_{(G,d)}: W^V_\No\circ G\to G\circ W^V_\Mo,$$
$$(\tau_{(G,d)})_M: V\overline{\otimes}G(M)\to G(V\overline{\otimes} M), (\tau_{(G,d)})_M=d^{-1}_{V,M},$$
for any $M\in\Mo$. Then, $\tau_{(G,d)}$ is a $\ca$-module natural isomorphism.

Now, we shall define the functor $\Fc$ on morphisms. Let $(V,\gamma)$, $(V',\gamma')$ be objects in $\Zc_\ca(\Do)$ and $f:(V,\gamma)\to (V',\gamma')$ be an arrow in $\Zc_\ca(\Do)$. Define
$\Fc(f):(W^V,\tau)\to (W^{V'},\tau'),$ as follows. For any  $\Do$-module  $\Mo$,  define the $\ca$-module natural transformation $$\Fc(f)_\Mo: W^V_\Mo\to W^{V'}_\Mo, \quad (\Fc(f)_\Mo)_M=f\overline{\otimes}\id_M,$$
for any $M\in\Mo$.
\bigbreak

Now, we shall define a functor $\Gc:\Zc(\Hc)\to \Zc_\ca(\Do),$ that will be the inverse of $\Fc$. Any object $X\in \ca$ induces a $\Do$-module functor $J_X:\Do\to\Do$, $J_X(V)=V\otimes X$. 
\medbreak

Let $(W,\tau)$ be an object in $\Zc(\Hc)$. For any $\Do$-module category $\Mo$, $W_\Mo:\Mo\to \Mo$ is a $\ca$-module functor. We shall denote it by $W_\Mo=(W_\Mo, c^\Mo)$. In particular,  $W_\Do(\uno)\in \Do$. We have natural $\ca$-module isomorphisms
$(\tau_{\Do,\Do})_{J_X}:W_\Do\circ J_X\xrightarrow{\backsimeq} J_X\circ W_\Do.$ In particular, we have isomorphisms 
$$((\tau_{\Do,\Do})_{J_X})_\uno:W_\Do(X)\xrightarrow{\backsimeq} W_\Do(\uno)\otimes X.$$
Using that $W_\Do$ has a $\ca$-module structure, there is a natural isomorphism $$c^\Do_{X,\uno}:X\otimes W_\Do(\uno)\to W_\Do(X).$$ Let $\gamma$ be the natural isomorphism defined as
$$\gamma_X: X\ot W_\Do(\uno)\to W_\Do(\uno)\ot X, \quad\gamma_X=((\tau_{\Do,\Do})_{J_X})_\uno\circ c_{X,\uno}.$$ 
The natural transformation $\gamma$ satisfies \ref{bien} since $(\tau_{\Do,\Do})_{J_X}$ is a $\ca$-module natural transformation.
Then $(W_\Do(\uno),\gamma)\in \Zc_\ca(\Do).$ Whence, we define $\Gc(W,\tau)=(W_\Do(\uno),\gamma)$.

Let $f:(W,\tau)\to (W',\tau')$ be a morphism in $\Zc(\Hc)$, then $(f_\Do)_\uno$ is a morphism in $\Zc_\ca(\Do)$ since $f_\Do$ is a $\ca$-module natural transformation. Set
$\Gc(f)=(f_\Do)_\uno.$ It follows straightforward that $\Gc$ is well-defined and that $\Fc$ and $\Gc$ are inverse of each other.
\end{proof}

The center of the 2-category of representations of a tensor category $\ca$ coincides with the Drinfeld center of $\ca$.
\begin{cor}\label{center-modcat}
$\Zc({}_\ca \Mod)\backsimeq \Zc(\ca).$
\end{cor}

\begin{proof}
Take $\Do=\ca$ and $\Hc:{}_\ca \Mod\to {}_\ca \Mod$ the identity pseudofunctor. 
\end{proof}

\subsection{Group actions coming from graded tensor categories}

Throughout this section $G$ will denote a finite group. 
Assume that $\ca$ is a finite tensor category and  $\Do=\oplus_{g\in G} \Do_g$ is a $G$-graded extension of $\ca$. 
 Set $\Do_1=\ca$. We shall further assume that $\Do$ is a strict monoidal category.

In this section we aim to prove the following result.

\begin{teo}\label{g-act-graded-tc} There is an action of $G$ on the 2-category 
${}_{\ca}\Mod^{\opp}$. Moreover, there are 2-equivalences
$$ ({}_{\ca}\Mod^{\opp})^G \simeq {}_{\Do}\Mod, \quad ({}_{\ca}\Mod^{\opp}_e)^G \simeq {}_{\Do}\Mod_e .$$
\end{teo}
\pf First, let us define an action of $G$ on the 2-category $\Bc={}_{\ca}\Mod^{\opp}$. For any $g\in G$ define the 2-functors $F_g:\Bc\to \Bc$ as follows. For any left $\ca$-module category $\Mo$, set
$F_g(\Mo)=\Fun_{\ca}(\bca_g,\Mo)$. If $\Mo, \No$ are left
$\ca$-module categories, and 
$G:\Mo\to \No$ 
is a $\ca$-module functor, then 
$$F_g(G): \Fun_{\ca}(\bca_g,\Mo) \to \Fun_{\ca}(\bca_g,\No), \quad F_g(G)(H)=G\circ H.$$
Now, we shall define the pseudonatural equivalences 
$\chi_{g,h}: F_g\circ F_h\to F_{gh}$, for any $g, h\in G$. For any left $\ca$-module category $\Mo$
$$( \chi^0_{g,h})_{\Mo}: \Fun_{\ca}(\bca_{gh},\Mo) \to \Fun_{\ca}(\bca_g,\Fun_{\ca}(\bca_h,\Mo) ),$$
$$( \chi^0_{g,h})_{\Mo}(H)(X)(Y)=H(X\ot Y), $$
for any $H\in \Fun_{\ca}(\bca_{gh},\Mo)$, $X\in \ca_g, Y\in \ca_h$. It follows that $( \chi^0_{g,h})_{\Mo}$ is a well-defined $\ca$-module functor. For any $\ca$-module functor $G:\Mo\to \No$ we have that $F_g(F_h(G)\circ ( \chi^0_{g,h})_{\Mo}=( \chi^0_{g,h})_{\No}\circ F_{gh}(G)$, whence, we can define 
$$( \chi_{g,h})_{G}: F_g(F_h(G)\circ ( \chi^0_{g,h})_{\Mo}\to ( \chi^0_{g,h})_{\No}\circ F_{gh}(G)$$
to be the identities. Since $\chi_{gh,f} \circ (\chi_{g,h} \ot  \idn_{F_f }) = \chi_{g,hf}\circ (\idn_{F_g} \ot\chi_{h,f})$, for any $f,g,h\in G$, then we can choose
$\omega_{g,h,f}$ to be the identities.

Now, we shall define a biequivalence $\Phi: \Bc^G\to {}_{\Do}\Mod$. Assume $(\Mo, U, \Pi)$ is an equivariant 0-cell. This means that we have $\ca$-module functors
$$U_g:\Fun_{\ca}(\bca_g,\Mo)\to \Mo,$$
together with $\ca$-module natural isomorphisms
$$\Pi_{g,h}: U_g\circ F_g(U_h)\circ (\chi^0_{g,h})_{\Mo} \to U_{gh},$$
satisfying the required axioms. Recall the definition of the functors $G_{X,M}$ given in Section \ref{graded-tc}. 
\begin{claim} Let be $g, h\in G$. If $X\in\ca_g, Y\in \ca_h$, then, there exists a family of $\ca$-module natural isomorphisms 
$$\beta_{X,Y,M}:  F_g(U_h)\big((\chi^0_{g,h})_{\Mo}(G_{X\ot Y,M})\big)\to G_{X, U_h(G_{Y,M})} .$$
\end{claim}
\pf[Proof of Claim] If $Z\in\ca_g$, then
$$G_{X, U_h(G_{Y,M})}(Z)=({}^*Z\ot X)\otb U_h(G_{Y,M}), $$
$$F_g(U_h)\big((\chi^0_{g,h})_{\Mo}(G_{X\ot Y,M})\big)(Z)=U_h(G_{X\ot Y,M}(Z\ot -)). $$
Note that there are module natural isomorphisms
$$G_{X,M}(Z\ot -)\simeq {}^*Z\otb G_{X,M}, \quad X\otb G_{Y,M}\simeq G_{X\ot Y,M}.$$
Combining these two isomorphisms we get that
$$G_{X\ot Y,M}(Z\ot -) \simeq ({}^*Z\ot X)\otb G_{Y,M}.$$
Using this isomorphism and the fact that $U_h$ is a $\ca$-module functor, we get that
$$U_h(G_{X\ot Y,M}(Z\ot -))\simeq ({}^*Z\ot X)\otb U_h(G_{Y,M}),$$
obtaining the desired isomorphisms.
\epf
We define $\Phi(\Mo, U, \Pi)=\Mo$ as Abelian categories. We must endowed the category $\Mo$ with a structure of $\Do$-module category. If $X\in \ca_g$, $M\in \Mo$ set
$$X\otb M= U_g( G_{X,M}).$$
We have to define associativity isomorphisms
$$ m_{X,Y,M}:(X\ot Y)\otb M\to X\otb  (Y \otb M).$$
Suppouse that $X\in \ca_g, Y\in\ca_h$, $M\in \Mo$. Then
$$(X\ot Y)\otb M = U_{gh}( G_{X\ot Y,M}), \quad X\otb  (Y \otb M)=U_g(G_{X, U_h(G_{Y,M})}).$$
Hence, we define
$$ m_{X,Y,M}= U_g(\beta_{X,Y,M}) (\Pi_{g,h})^{-1}_{G_{X\ot Y,M}}.$$
Axiom \eqref{left-modulecat1} is equivalent, in this case, to axiom \eqref{a-pi-eq}. It is clear that $\Phi$ is a biequivalence and restricted to the category of exact modules $({}_{\ca}\Mod^{\opp}_e)$ gives the second biequivalence.
\epf 

\section{Braided $G$-crossed tensor categories from $G$ actions on 2-categories}\label{braided-Gcrossed}

In this section actions of groups on 2-categories are assumed to be strict. This does not lead to any loss of
generality, since, in view of Theorem \ref{coherence-groupact}, all definitions and statements remain valid for non-strict actions after insertion of the suitable  isomorphisms.

\subsection{Strict braided $G$-crossed tensor categories}

Braided $G$-crossed fusion categories play the same role in homotopy quantum field theory that braided fusion categories in the topological quantum field theory, see \cite{T1,T2,T3}. 

\begin{defi}
Let $G$ be a groups and $\Cc$ a strict monoidal category. A \textit{strict} braided $G$-crossed structure on $\Cc$ consist of the following data:

\begin{enumerate}
 \item a decomposition $\Cc=\coprod_{g\in G} \Cc_g$  (coproduct of categories) such that
 \begin{itemize}
 \item $\mathbf{1} \in \Cc_e$,
 \item $\Cc_g\otimes \Cc_h \subset \Cc_{gh}$ for all $g,h\in G$,
 \end{itemize}
 \item a $G$-indexed family of strict monoidal functor $g_*:\Cc\to \Cc$, such that
 \begin{itemize}
 \item $g_*(\Cc_h)\subset \Cc_{ghg^{-1}}$,\quad$g_*h_*=(gh)_*$,\quad $e_*=\operatorname{Id}_\Cc$,
 \end{itemize}
\item  a family of  natural isomorphisms

\[
\xymatrixrowsep{0.6in}
\xymatrixcolsep{0.6in}
\xymatrix{
\Cc\times \Cc_g  \ar[rr]^{ g_*\times \operatorname{Id}_{\Cc}}&& \Cc\times \Cc_g \ar[d]^{\otimes}\\ \ar[u]^{\operatorname{flip}}
\Cc_{g}\times \Cc \ar[rr]_{\otimes} && \Cc  {\ar @{} [llu] |{\Downarrow c}}
}
\]

such that
 \begin{itemize}
    \item $g_*(c_{X,Z})=c_{g_*(X),g_*(Z)}$ 
    \item $c_{X,Y\ot Z}=(\id_Y\ot c_{X,Z})\circ (c_{X,Y}\ot \id_{Z})$
    \item $c_{X\ot Y,Z}=(c_{X,h_*(Z)}\ot \id_{Y}) \circ (\id_X\ot c_{Y,Z})$
\end{itemize}
\end{enumerate}
for all $X\in \Cc$, $Y\in \Cc_g, Z\in \Cc_h,$ $g,h\in G$.
\end{defi}

Even when the  definition of strict braided $G$-crossed monoidal category is too restrictive, every \emph{weak} braided $G$-crossed category is equivalent to a \emph{strict} braided $G$-crossed category, see \cite{Ga2}.

\subsection{Center of a $G$-action}
Let $G$ be a group acting strictly on a 2-category $\bb$, where $F_g:\bb\to \bb$, denotes the associated 2-functors. We shall introduce a $G$-graded monoidal category equipped with an action of $G$.
\subsubsection{The $G$-graded monoidal category $\mathcal{Z}_G(\bb)$}\label{definition of  Z_G}
Define the strict monoidal category $\mathcal{Z}_G(B)=\coprod_{g\in G}\mathcal{Z}_G(B)_g,$ where $\mathcal{Z}_G(B)_g= \operatorname{Pseu-Nat}(\operatorname{Id}_{\bb},F_g)$ and  the product induced by the tensor product of pseudonatural transformation defined in \eqref{tensorp-psnat}.  In other words, if $X\in \mathcal{Z}_G(B)_g$ and $ Y\in \mathcal{Z}_G(B)_h$, we define $X\otimes Y\in \mathcal{Z}_G(B)_{gh}= \operatorname{Pseu-Nat}(\operatorname{Id}_{\bb},F_{gh})$ as folows: for any object $A\in \bb$, $(X\otimes B)_A=X_{F_h(A)}\circ Y_A$  and for any 1-cell $W\in \bb(A,B)$ 

\[
\xymatrixrowsep{0.6in}
\xymatrixcolsep{0.6in}
\xymatrix{
A \ar@/_3pc/[dd]_{(X\otimes Y)_A}  \ar[d]^{Y_A} \ar[rr]^{W} && B \ar[d]^{Y_B} \ar@/^3pc/[dd]^{(X\otimes Y)_B}\\ 
 F_h(A) \ar[rr]^{F_h(W)} \ar[d]^{X_{F_h(A)}}  &&   F_h(B) \ar[d]^{F_h(B)}{\ar @{} [llu] |{\Downarrow Y_W}} \\
 F_{gh}(A) \ar[rr]^{X_{F_h(W)}}  && F_{gh}(B) {\ar @{} [llu] |{\Downarrow X_{F_h(X)}}}
}
\] 
The unit object is $1_{\operatorname{Id}_{\bb}}\in \operatorname{Pseu-Nat}(\operatorname{Id}_{\bb},\operatorname{Id}_{\bb})$.

\subsubsection{The action of $G$ on $\mathcal{Z}_G(\bb)$}\label{definition of action on Z_G}
Given $X\in \mathcal{Z}_G(B)_h$ and $g\in G$, we define $g_*(X)\in \mathcal{Z}_G(B)_{ghg^{-1}}$ as follows: for objects $A\in \bb$, $g_*(X)_A=F_g(X_{F_{g^{-1}}(A) } )$ and  for any 1-arrow $W:A\to B$

\[
\xymatrixrowsep{0.6in}
\xymatrixcolsep{0.6in}
\xymatrix{
A \  \ar[d]_{F_g(X_{F_{g^{-1}}(A) } ) } \ar[rr]^{W} && B \ar[d]^{F_g(X_{F_{g^{-1}}(B) } )} \\ 
 F_{ghg^{-1}}(A) \ar[rr]^{F_{ghg^{-1}}(W)}   &&   F_{ghg^{-1}}(B). {\ar @{} [llu] |{\Downarrow F_g(X_{F_{g^{-1}}(W) } )}}
}
\]Analogously, the functor $g_*$ is defined for morphism in $\mathcal{Z}_G(B)$. 

\subsubsection{The $G$-braiding  of $\mathcal{Z}_G(\bb)$}\label{definition of G-braiding}
Let $X\in \mathcal{Z}_G(B)_g$ and $ Y\in \mathcal{Z}_G(B)_h$. By the definition of pseudo-natural transformation we have 
\[
\xymatrixrowsep{0.6in}
\xymatrixcolsep{0.6in}
\xymatrix{
A \  \ar[d]_{X_A} \ar[rr]^{Y_A} && F_h(A) \ar[d]^{X_{F_h(A)}} \\ 
 F_{g}(A) \ar[rr]^{F_{g}(Y_A)}   &&   F_{gh}(A), {\ar @{} [llu] |{\Downarrow X_{Y_A}  }}
}
\]but $(X\otimes Y)_A= X_{F_h(A)}\circ Y_A$ and $(g_*(Y)\otimes X)_A= F_g(Y_A)\circ X_A$, then the $X_{Y_A}$ define natural isomorphism $c_{X,Y}:= X_{Y_A}: X\otimes Y\to g_*(Y)\otimes X$.

\begin{teo}
Let $G$ be a groups with a strcit action on a 2-categoy $\bb$. Then the monoidal category $\mathcal{Z}_G(\bb)$ defined in \ref{definition of Z_G} is a strict braided  $G$-crossed monoidal category with action defined in \ref{definition of action on Z_G} and $G$-braiding defined in \ref{definition of G-braiding}. Moreover, the braided category $\mathcal{Z}_G(\bb)_e$ is exactly the Drinfeld center of $\bb$.
\end{teo}
\begin{proof}
Since the action of $G$ on $\bb$ is strict, it follows by definition the equations  \begin{itemize}
    \item $g_*(c_{X,Z})=c_{g_*(X),g_*(Z)}$ 
    \item $c_{X,Y\ot Z}=(\id_Y\ot c_{X,Z})\circ (c_{X,Y}\ot \id_{Z})$
    \item $c_{X\ot Y,Z}=(c_{X,h_*(Z)}\ot \id_{Y}) \circ (\id_X\ot c_{Y,Z})$.
\end{itemize}
\end{proof}

\subsection{Example}
Let $\Do=\oplus_{g\in G}\Do_g$ be a faithfully $G$-graded fusion category.

Since every $\Do_g$ is a $\Do_e$-bimodule category, they define 2-functors \[F_g(-):=\Do_g\boxtimes_{\Do_e}(-): \Do_e-\operatorname{Mod}\to \Do_e-\operatorname{Mod},\] the  tensor products $\otimes:\Do_g\times \Do_h\to \Do_{gh}$ induce pseudo-natural equivalences 
$\chi_{g,h}: F_g\circ F_h\to F_{gh}$ and the associator of $\Do$ induce invertible modifications $\omega_{g,h,f}: \chi_{gh,f} \circ (\chi_{g,h} \ot  \idn_{F_f }) \Rightarrow \chi_{g,hf}\circ (\idn_{F_g} \ot\chi_{h,f}),$ that defines an action of $G$ on $\Do_e-\operatorname{Mod}$.  See \cite{ENO3} for details.

In this case the category $\mathcal{Z}_G({}_{\Do_e}\Mod)_g$ is just $\Fun_{\Do_e-\Do_e}(\Do_e,\Do_g)$, the category of $\Do_e$-bimodule functors and natural transformations from $\Do_e$ to $\Do_g$. 
The category $\mathcal{Z}_G({}_{\Do_e}\Mod)_g$ is canonically equivalent to the category $\mathcal{Z}_{\Do_e}(\Do_g)$ defined in \cite[Definition 2.1]{GNN} (use that  $\Fun_{\Do_e}(\Do_e,\Do_g)\to \Do_g, F\mapsto F(\mathbf{1})$  is a category equivalence). Then the $G$-graded category $\mathcal{Z}_G({}_{\Do_e}\Mod)$ is equivalent to the monoidal category $\mathcal{Z}_{\Do}(\Do_e)$. The braided $G$-crossed category $\mathcal{Z}_G({}_{\Do_e}\Mod)$ is equivalent to the $G$-crossed category $\mathcal{Z}_{\Do}(\Do_e)$ defined in \cite{GNN}.

\section{The center of the equivariant 2-category}\label{Section-center}

This section is devoted to prove the following result.
Let $G$ be a finite group acting on
a 2-category $\Bc$.  Recall the forgetful 2-functor $\Phi:\Bc^G\to \Bc$ described in Lemma \ref{forget-2funct}.

\begin{teo}\label{center-equi} The group $G$ acts on $\Zc(\Phi)$ by monoidal autoequivalences, and there is a  monoidal equivalence
$$ \Zc(\Bc ^G)\simeq \Zc(\Phi)^G.$$
\end{teo}
As a consequence, we have the following result. 
\begin{cor}\cite[Thm. 3.5]{GNN} Let $\Do=\oplus_{g\in G} \ca_g$ be a faithfully graded tensor category, with $\ca=\ca_1$. There is an action of the group $G$ on the relative center $\Zc_{\ca}(\Do)$ and a monoidal equivalence
$$\Zc(\Do) \simeq \Zc_{\ca}(\Do)^G.$$
\end{cor}
\pf Let $\Hc: {}_\Do\Mod\to  {}_{\ca}\Mod$ be the forgetful pseudofunctor. Then
$$\Zc(\Do) \simeq \Zc({}_{\Do}\Mod) \simeq \Zc(\big({}_{\ca}\Mod^{\opp}\big)^G) \simeq \Zc(\Hc)^G \simeq \Zc_{\ca}(\Do)^G$$
The first equivalence follow from Corollary \ref{center-modcat}, the second one is Theorem \ref{g-act-graded-tc}, and the last one is Proposition \ref{cent-forget-relative-center}.
\epf

For the rest of this section we shall use the notation introduced in Section \ref{section:unpacking}. There is no harm in assuming that the action is \textit{unital} and \textit{strict}, see definitions \ref{unital-act}, \ref{strict-act}. By Proposition \ref{equivalence-1cells-iso}, we can assume that  any invertible 1-cell is an isomorphism. In particular, if $(A,U,\Pi)$ is an equivariant 0-cell, for any $g\in G$, the 1-cell $U_g$ is invertible. Thus, we can choose a 1-cell $U^*_g$ such that 
$$U_g\circ U^*_g=I_{F_g(A)}, \quad U^*_g\circ U_g=I_A.$$

If $X, Y$ are 1-cells, we shall sometimes denote $X\circ Y=XY$, as a space saving measure.

\subsection{A group action on $\Zc(\Phi)$} For any $g\in G$, we shall define tensor autoequivalences $L_g:\Zc(\Phi)\to \Zc(\Phi)$ such that they define an action of $G$ on $ \Zc(\Phi).$ First, let us explicitly describe objects in $\Zc(\Phi)$. An object $(X,\sigma)\in \Zc(\Phi)$ consists of
$$X=\{X_{(A,U,\Pi)}\in \Bc(A,A) \,\text{ a 1-cell}, (A,U,\Pi)\in \Obj(\Bc^G)\},$$
$$\sigma=\{\sigma_{(\theta,\theta_g)}: X_{(\widetilde A,\widetilde U,\widetilde\Pi)}\circ \theta \Rightarrow \theta\circ X_{(A,U,\Pi)}\, \text{ isomorphisms 2-cells in}\, \Bc^G\}, $$
where $(\theta,\theta_g)\in \Bc^G((A,U,\Pi), (\widetilde A,\widetilde U,\widetilde\Pi))$ is an equivariant 1-cell. The isomorphisms $\sigma_{(\theta,\theta_g)}$ satisfy \eqref{relativ-cent1}. If  $(X,\sigma), (Y,\tau)\in \Zc(\Phi)$, a morphism $f:(X,\sigma)\to (Y,\tau)$ is a collection  of 2-cells in $\Bc(A,A)$
$$f_{(A,U,\Pi)}: X_{(A,U,\Pi)} \Rightarrow Y_{(A,U,\Pi)},$$
such that for any equivariant 1-cell $(\theta,\theta_g)\in \Bc^G((A,U,\Pi), (\widetilde A,\widetilde U,\widetilde\Pi))$
$$\big( \id_{\theta}\circ  f_{(A,U,\Pi)}\big) \sigma_{(\theta,\theta_g)}= \tau_{(\theta,\theta_g)}\big(f_{(\widetilde A,\widetilde U,\widetilde\Pi)} \circ \id_{\theta}\big). $$
\begin{lema}\label{some-morphisms} Suppose $g, h\in G$ and $(A,U,\Pi)$ is an equivariant 0-cell.  There are isomorphisms 2-cells 
$$\epsilon_{g,h,(A,U,\Pi)}: U^*_g\circ F_g(U^*_h)\Rightarrow U^*_{gh}$$
such that 
\begin{equation}\label{epsilon-pi}\epsilon_{g,h,(A,U,\Pi)}\circ \Pi_{g,h}=\id_{I_A}, \quad \Pi_{g,h}\circ\epsilon_{g,h,(A,U,\Pi)}=\id_{I_{F_{gh}(A)}},
\end{equation} 
\begin{equation}\label{epsilon-pi2} \epsilon_{gh,f,(A,U,\Pi)} \big( \epsilon_{g,h,(A,U,\Pi)}\circ \id_{F_{gh}(U^*_f)} \big)=\epsilon_{g,hf,(A,U,\Pi)} \big( \id_{U^*_g}\circ F_g(\epsilon_{h,f,(A,U,\Pi)}) \big),
\end{equation}
for any $g,h,f\in G$.
\end{lema}
\pf Take $\epsilon_{g,h,(A,U,\Pi)}=\id_{U^*_g\circ F_g(U^*_h)}\circ \Pi^{-1}_{g,h}\circ \id_{U^*_{gh}}.$  Equation \eqref{epsilon-pi2} follow from \eqref{a-pi-eq}.
\epf
For any $g\in G$, let us define the functors $L_g:\Zc(\Phi)\to \Zc(\Phi)$, $L_g(X,\sigma)=(X^g,\sigma^g)$. Where, for any equivariant 0-cell $(A,U,\Pi)$
$$X^g_{(A,U,\Pi)}=U^*_g\circ F_g(X_{(A,U,\Pi)})\circ U_g.$$
\begin{rmk} As a saving space measure, if $(A,U,\Pi), (\widetilde A,\widetilde U,\widetilde\Pi)$ are equivariant 0-cells, we are going to denote $X=X_{(A,U,\Pi)}$, $\widetilde X=X_{(\widetilde A,\widetilde U,\widetilde\Pi)}$. Also, we shall denote 
$\epsilon_{g,h}=\epsilon_{g,h,(A,U,\Pi)}$ and $\widetilde \epsilon_{g,h}=\epsilon_{g,h,(\widetilde A,\widetilde U,\widetilde\Pi)}$ when no confusion arises. 
\end{rmk}
If $(\theta,\theta_g)\in \Bc^G((A,U,\Pi),(\widetilde A,\widetilde U,\widetilde\Pi))$ is an equivariant 1-cell, then
$$\sigma^g_{(\theta,\theta_g)}=\big( 1_{\widetilde U^{*}_g}\circ \theta_g\circ 1_{U^*_gF_g(X)U_g} \big)\big(1_{\widetilde  U^*_g}\circ F_g(\sigma_{(\theta,\theta_g)}) \circ1_{U_g} \big)\big(1_{\widetilde U_gF_g(\widetilde X)} \circ \theta^{-1}_g\big).$$
If $f:(X,\sigma)\to (Y,\tau)$ is a morphism in $\Zc(\Phi)$, then
\[L_g(f)_{(A,U,\Pi)}=\id_{U^*_g}\circ F_g(f_{(A,U,\Pi)})\circ\id_{U_g}.\]
The proof of the next result follows  straightforwardly. 
\begin{prop} The functors $L_g:\Zc(\Phi)\to \Zc(\Phi)$ are well-defined monoidal functors. \qed
\end{prop}
Now, for any $g,h\in G$, we shall define monoidal natural isomorphisms $\nu_{g,h}: L_g\circ L_h\to L_{gh}$ satisfying \eqref{group-act-tc1} and \eqref{group-act-tc2}.  Take $(X,\sigma)\in  \Zc(\Hc)$, so we must define an arrow 
$$ (\nu_{g,h})_{(X,\sigma)}: L_g\circ L_h(X,\sigma)\to L_{gh}(X,\sigma).$$
For each equivariant 0-cell $(A,U,\Pi)$ we  define the map
$$\!\big((\nu_{g,h})_{(X,\sigma)} \big)_{(A,U,\Pi)}: U^{*}_gF_g(U^{*}_h)F_{gh}(X_{(A,U,\mu)})F_g(U_{h})U_{g}\to U_{gh}F_{gh}(X_{(A,U,\mu)})U^{*}_{gh},$$
$$\big((\nu_{g,h})_{(X,\sigma)} \big)_{(A,U,\mu)}=\epsilon_{g,h}\circ\id_{F_{gh}(X_{(A,U,\mu)})}\circ \Pi_{g,h}. $$
\begin{prop} For any $g,h,f\in G$, the following assertions holds.
\begin{itemize}
\item[(i)] $\nu_{g,h}: L_g\circ L_h\to L_{gh}$ are well-defined natural isomorphisms in $\Zc(\Phi)$.
\item[(ii)] $\nu_{g,h}: L_g\circ L_h\to L_{gh}$ are monoidal natural transformations.

\item[(iii)] For any  $g,h,f\in G$ and any $(X,\sigma)\in  \Zc(\Phi)$, the following equation holds
\begin{equation}\label{assoc-mon-l}
(\nu_{gh,f})_{(X,\sigma)}(\nu_{g,h})_{L_f(X,\sigma)}=(\nu_{g,hf})_{(X,\sigma)} L_g((\nu_{h,f})_{(X,\sigma)}).
\end{equation} 
\end{itemize}
\end{prop}
\pf (i). We must verify  that $(\nu_{g,h})_{(X,\sigma)}$ are morphisms in the category $\Zc(\Phi)$, that is, equation
\begin{equation}\label{nu-morhisms-incat} 
\big(\id_{\theta} \circ  \big((\nu_{g,h})_{(X,\sigma)} \big)_{(A,U,\mu)}\big) ((\sigma^h)^g)_{(\theta,\theta_g)}=\sigma^{gh}_{(\theta,\theta_g)} \big( \big((\nu_{g,h})_{(X,\sigma)} \big)_{(\widetilde A,\widetilde U,\widetilde\Pi)}\circ\id_{\theta} \big)
\end{equation}
is fulfilled for any equivariant 1-cell $(\theta,\theta_g)\in \Bc^G((A,U,\Pi), (\widetilde A,\widetilde U,\widetilde\Pi))$. The left hand side of \eqref{nu-morhisms-incat} equals to
\begin{align*}&= \big( \id_\theta\circ\epsilon_{g,h}\circ\id_{F_{gh}(X)}\circ \Pi_{g,h} \big)\big(\id_{\widetilde U^{*}_g}\circ \theta_g\circ\id_{ U^{*}_g F_g(U^{*}_h)F_{gh}(X)F_g(U_h)U_g}  \big)\\
&\big( \id_{\widetilde U^{*}_g}\circ F_g(\sigma^h_{(\theta,\theta_g)})\circ\id_{U_g} \big)\big(\id_{\widetilde U^{*}_g}\circ \theta_g\circ \id_{U^{*}_gF_g(U^{*}_h)F_{gh}(X)F_g(U_h)U_g}  \big)\\
&= \big( \id_\theta\circ\epsilon_{g,h}\circ\id_{F_{gh}(X)}\circ \Pi_{g,h} \big)\big(\id\circ \theta_g\circ\id \big)
\big( \id_{\widetilde U^{*}_g F_g(\widetilde U^{*}_h)}\circ F_g(\theta_h) \circ\id\big)\\ & \big( \id_{\widetilde U^{*}_g F_g(\widetilde U^{*}_h)}\circ F_{gh}(\sigma_{(\theta,\theta_g)})\circ\id_{F_g(U_h)U_g} \big)\big(  \id_{\widetilde U^{*}_g F_g(\widetilde U^{*}_h)}\circ F_g(\theta^{-1}_h)\circ\id_{U_g}\big)\\&\big(\id_{\widetilde U^{*}_g}\circ \theta_g\circ \id_{U^{*}_gF_g(U^{*}_h)F_{gh}(X)F_g(U_h)U_g}  \big)\\
&=\big( \id\circ \epsilon_{g,h}\circ \id \big)\big(\id_{\widetilde U^{*}_g F_g(\widetilde U^{*}_h)}\circ (\id_{F_g(\widetilde U_h)}\circ \theta_g)(F_g(\theta_h)\circ \id_{U_g})\circ \id_{U^{*}_gF_g(U^{*}_h)F_{gh}(X)U_{gh}}  \big)\\&\big( \id_{\widetilde U^{*}_g F_g(\widetilde U^{*}_h)}\circ F_{gh}(\sigma_{(\theta,\theta_g)})\circ\id_{U_{gh}} \big)\\
&\big(\id_{\widetilde U^{*}_g F_g(\widetilde U^{*}_h)F_{gh}(\widetilde X)}\circ (\id_{F_{gh}(\theta)}\circ \Pi_{g,h})(F_g(\theta^{-1}_h)\circ \id_{U_g})(\id_{F_{gh}(\widetilde U_h)}\circ \theta^{-1}_g)  \big)\\
&=\big( \id\circ \epsilon_{g,h}\circ \id \big)\big(\id_{\widetilde U^{*}_g F_g(\widetilde U^{*}_h)}\circ (\id_{F_g(\widetilde U_h)}\circ \theta_g)(F_g(\theta_h)\circ \id_{U_g})\circ \id  \big)\\&\big( \id_{\widetilde U^{*}_g F_g(\widetilde U^{*}_h)}\circ F_{gh}(\sigma_{(\theta,\theta_g)})\circ\id_{U_{gh}} \big) \big( \id\circ \theta^{-1}_{gh} (\widetilde \Pi_{g,h}\circ \id_\theta) \big)
\end{align*}

The second equation follows from the definition of $\sigma^h_{(\theta,\theta_g)}$, the fourth equality follows from \eqref{equiva-1cell}. The right hand side of \eqref{nu-morhisms-incat} equals to
\begin{align*}&=\big(\id_{\widetilde U^{*}_{gh}} \circ \theta_{gh}\circ \id_{U^{*}_{gh}F_{gh}(X) U_{gh}} \big)\big( \id_{\widetilde U^{*}_{gh}}\circ F_{gh}(\sigma_{(\theta,\theta_g)})\circ  \id_{U_{gh}} \big)\\& \big(\id_{\widetilde U^{*}_{gh} F_{gh}(\widetilde X)}\circ  \theta^{-1}_{gh} \big)\big(\widetilde  \epsilon_{g,h}\circ  \id_{F_{gh}(\widetilde X)} \circ \widetilde \Pi_{g,h}\circ \id_{\theta}\big)\\
&=\big(\widetilde  \epsilon_{g,h}\circ \theta_{gh}\circ\id_{U^{*}_{gh}F_{gh}(X)U_{gh} }  \big)\big( \id_{U^{*}_gF_g(U^{*}_h)}\circ F_{gh}(\sigma_{(\theta,\theta_g)})\circ\id_{U_{gh}} \big)\\ &\big(\id_{U^{*}_gF_g(U^{*}_h)F_{gh}(\widetilde X)}\circ \theta^{-1}_{gh} ( \widetilde \Pi_{g,h}\circ \id_\theta) \big).
\end{align*}
It follows from Equation \eqref{epsilon-pi} that both sides are equal.
\medbreak

(ii). Let $(X,\sigma), (Y,\tau)$ be objects in $\Zc(\Phi)$. Since the functors $L_g$ are strict, this means that $L_g((X,\sigma)\ot (Y,\tau))=L_g(X,\sigma)\ot  L_g(Y,\tau)$, we must prove that
\begin{equation}\label{nu-mon1}  (\nu_{g,h})_{(X,\sigma)\ot (Y,\tau)}= (\nu_{g,h})_{(X,\sigma)}\ot (\nu_{g,h})_{(X,\sigma)}.
\end{equation}
Let $(A,U,\Pi)$ be an equivariant 0-cell. The left hand side of \eqref{nu-mon1} evaluated in $(A,U,\Pi)$ equals to
$$\epsilon_{g,h}\circ \id_{F_{gh}(X_{(A,U,\Pi)})}\circ\Pi_{g,h}\circ\epsilon_{g,h}\circ  \id_{F_{gh}(Y_{(A,U,\Pi)})}\circ \Pi_{g,h}.$$
The right hand side of \eqref{nu-mon1} evaluated in $(A,U,\Pi)$ equals to
$$\epsilon_{g,h}\circ \id_{F_{gh}(X_{(A,U,\Pi)}\circ Y_{(A,U,\Pi)})}\circ\Pi_{g,h} .$$
It follows from \eqref{epsilon-pi} that both sides are equal.
\medbreak

(iii). Let $(A,U,\Pi)$ be an equivariant 0-cell. The left hand side of \eqref{assoc-mon-l} evaluated in $(A,U,\Pi)$ is equal to
\begin{align*}&=\big(\epsilon_{gh,f}\circ\id_{F_{ghf}(X)}\circ \Pi_{gh,f} \big)\big( \epsilon_{g,h}\circ\id_{F_{gh}(U^*_fXU_f)}\circ \Pi_{g,h}\big)\\
&=\epsilon_{gh,f} (\epsilon_{g,h}\circ\id_{F_{gh}(U^*_f)} )\circ \id_{F_{ghf}(X)}\circ \Pi_{gh,f} ( \id_{F_{gh}(U_f)}\circ \Pi_{g,h}).
\end{align*}
The right hand side of \eqref{assoc-mon-l} evaluated in $(A,U,\Pi)$ is equal to
\begin{align*}&= \big( \epsilon_{g,hf} \circ \id_{F_{gh}(X)}\circ \Pi_{g,hf}\big)\big(\id_{U^*_g}\circ F_g(\epsilon_{h,f}) \circ\id_{F_{ghf}(X)}\circ F_g(\Pi_{h,f})\circ\id_{U_g} \big)\\
&=\epsilon_{g,hf} (\id_{U^*_g}\circ F_g(\epsilon_{h,f}))\circ \id_{F_{ghf}(X)} \circ \Pi_{g,hf}(F_g(\Pi_{h,f})\circ\id_{U_g}).
\end{align*}
Now, that both expressions are equal follow by \eqref{epsilon-pi2} and \eqref{a-pi-eq}.
\epf
\subsubsection{Proof of Theorem \ref{center-equi}}
Let us first describe an object in the equivariantization of the category $ \Zc(\Phi).$ An object in  $ \Zc(\Phi)^G$  is a collection $((X,\sigma),s)$ where $(X,\sigma)\in \Zc(\Phi)$, and $s_g: L_g(X,\sigma)\to (X,\sigma)$ is a morphism in the category, for any $g\in G$. This means, that $X_{(A,U,\Pi)}\in \Bc(A,A)$ is a 1-cell, for any equivariant 0-cell $(A,U,\Pi)$, and for any equivariant 1-cell $(\tau,\tau_g)\in \Bc^G((A,U,\Pi), (\widetilde A,\widetilde U,\widetilde\Pi))$ there is an isomorphism $\sigma_{(\tau,\tau_g)}: X_{(\widetilde A,\widetilde U,\widetilde\Pi)}\circ \tau\to \tau\circ X_{(A,U,\Pi)}$ such that equation \eqref{relativ-cent1} is fulfilled. Also, for any $g\in G$ and any equivariant 0-cell $(A,U,\Pi)$
there are  morphisms 
$$(s_g)_{(A,U,\Pi)}: U^*_g F_g(X_{(A,U,\Pi)}) U_g\to V_{(A,U,\Pi)},$$
such that 
\begin{equation}\label{equivariant-ob-s0} 
\big( \id_\tau\circ (s_g)_{(A,U,\Pi)}\big)\sigma^g_{(\tau,\tau^1)}= \sigma_{(\tau,\tau^1)}\big((s_g)_{(\widetilde A,\widetilde U,\widetilde\Pi)}\circ \id_\tau \big),
\end{equation}
\begin{equation}\label{equivariant-ob-s}
(s_{gh})_{(A,U,\Pi)} (\nu_{g,h})_{(A,U,\Pi)} = (s_g)_{(A,U,\Pi)}  L_g((s_h)_{(A,U,\Pi)}),
\end{equation}
for any equivariant 0-cells $(A,U,\Pi),  (\widetilde A,\widetilde U,\widetilde\Pi) $, any equivariant 1-cell $(\tau,\tau_g)\in \Bc^G((A,U,\Pi), (\widetilde A,\widetilde U,\widetilde\Pi))$, and any $g, h\in G$. Equation \eqref{equivariant-ob-s0} follows from the fact that $s_g: L_g(V,\sigma)\to (V,\sigma)$ is a morphism in the category $\Zc(\Phi)$, and equation \eqref{equivariant-ob-s} follows from \eqref{group-act-tc3}.
\medbreak

Define the functor $\Psi:\Zc(\Phi)^G\to \Zc(\Bc^G)$ as follows. Let $((X,\sigma),s)\in\Zc(\Phi)^G $, then  $\Phi((X,\sigma),s)=(V,\widetilde \sigma)$. For any equivariant 0-cell $(A,U,\Pi)$, $V_{(A,U,\Pi)}$ must be an equivariant 1-cell in the category $\Bc^G((A,U,\Pi),(A,U,\Pi))$. Define $V_{(A,U,\Pi)}=(X_{(A,U,\Pi)}, \theta^{(A,U,\Pi)}_g)$, where
\begin{align}\label{def-thetag}\begin{split}\theta^{(A,U,\Pi)}_g: F_g(X_{(A,U,\Pi)})\circ U_g \Rightarrow U_g\circ X_{(A,U,\Pi)},\\
\theta^{(A,U,\Pi)}_g=\id_{U_g}\circ  (s_g)_{(A,U,\Pi)}.\end{split}
\end{align} 
If $(\tau,\tau_g)\in \Bc^G((A,U,\Pi), (\widetilde A,\widetilde U,\widetilde\Pi))$ is an equivariant 1-cell, then 
$$\widetilde \sigma_{(\tau,\tau_g)}: (X_{(\widetilde A,\widetilde U,\widetilde\Pi)}, \theta^{(\widetilde A,\widetilde U,\widetilde\Pi)}_g)\circ (\tau,\tau_g) \Rightarrow (\tau,\tau_g) \circ (X_{(A,U,\Pi)}, \theta^{(A,U,\Pi)}_g),$$
$$\widetilde \sigma_{(\tau,\tau_g)}= \sigma_{(\tau,\tau_g)}.$$
\begin{claim} The following statements hold.
\begin{itemize}
\item[(i)] $V_{(A,U,\Pi)}=(X_{(A,U,\Pi)}, \theta^{(A,U,\Pi)}_g)\in \Bc^G$, for any equivariant 0-cell $(A,U,\Pi)$.

\item[(ii)] The object $(V,\widetilde \sigma)$ belongs to the category $ \Zc(\Bc^G).$ In particular, the functor $\Psi$ is well-defined.

\item[(iii)] The functor $ \Psi:\Zc(\Phi)^G\to \Zc(\Bc^G)$ is an equivalence of categories, and it has a monoidal structure.
\end{itemize}
\end{claim}
\pf[Proof of Claim] (i). We must check that the maps $\theta^{(A,U,\Pi)}_g)$ satisfy \eqref{equiva-1cell}. In this case, we must prove that for any $g,h\in G$
$$ \big(\Pi_{g,h}\circ\id_{X_{(A,U,\Pi)}}  \big)\big( \id_{F_g(U_h)}\circ \theta^{(A,U,\Pi)}_g\big)\big( F_g(\theta^{(A,U,\Pi)}_h)\circ \id_{U_g}\big)$$
is equal to
$$\theta^{(A,U,\Pi)}_{gh}\big(\id_{F_{gh}(X_{(A,U,\Pi)})}  \circ \Pi_{g,h}\big).
$$
Using the definition of $\theta^{(A,U,\Pi)}_g$, we get that the first expression is equal to
\begin{align*}&\big(\Pi_{g,h}\circ\id_{X_{(A,U,\Pi)}}  \big)\big( \id_{F_g(U_h)U_g}\circ (s_g)_{(A,U,\Pi)}\big) \big( \id_{F_g(U_h)}\circ F_g((s_h)_{(A,U,\Pi)})\circ \id_{U_g}\big)\\
&=\big(\Pi_{g,h}\circ\id_{X_{(A,U,\Pi)}}  \big) \big( \id_{F_g(U_h)U_g}\circ (s_g)_{(A,U,\Pi)} (\id_{U^*_g\circ F_g((s_h)_{(A,U,\Pi)})}) \circ \id_{U_g}\big)\\
&=\big(\Pi_{g,h}\circ\id_{X_{(A,U,\Pi)}}  \big) \big(   \id_{F_g(U_h)U_g}\circ  (s_{gh})_{(A,U,\Pi)} (\nu_{g,h})_{(A,U,\Pi)} \big)\\
&=\big( \id_{U_{gh}}\circ (s_{gh})_{(A,U,\Pi)} \big)\big(  \Pi_{g,h}\circ (\nu_{g,h})_{(A,U,\Pi)}\big)= \theta^{(A,U,\Pi)}_{gh} \big(\id_{F_{gh}(X_{(A,U,\Pi)})}  \circ \Pi_{g,h}\big).
\end{align*}
The second equality follows from \eqref{equivariant-ob-s}, and the last one follows from \eqref{epsilon-pi}.
\medbreak

(ii). Since $\widetilde \sigma_{(\tau,\tau_g)}= \sigma_{(\tau,\tau_g)}$ for any equivariant 1-cell $(\tau,\tau_g)$, then   $\widetilde \sigma$ satisfy \eqref{relativ-cent1}. We must verify only that $\widetilde \sigma_{(\tau,\tau_g)}$ is an equivariant 2-cell, that is  \eqref{equiva-2cell} is satisfied. To simplify the notation, let us denote $\theta^{(A,U,\Pi)}_{g}=\theta_g, \theta^{(\widetilde A,\widetilde U,\widetilde\Pi)}=\widetilde\theta_g. $ In this particular case, using the composition of equivariant 1-cells given by \eqref{comp-equiva-1cell}, we have to prove that
\begin{equation}\label{sigma-morph-incat} \big(1_{\widetilde U_g}\circ \sigma_{(\tau,\tau_g)} \big)\big(\widetilde \theta_{g}\circ 1_\tau \big)\big( 1_{F_g(\widetilde X)} \circ \tau_g \big)=\big( \tau_g \circ 1_X \big)\big( 1_{F_g(\tau)}\circ  \theta_{g}\big)\big( F_g(\sigma_{(\tau,\tau_g)})\circ 1_{U_g}\big).
\end{equation}
The left hand side of equation \eqref{sigma-morph-incat} is equal to
\begin{align*}&=\big( 1_{\widetilde U_g}\circ \sigma_{(\tau,\tau_g)} \big)\big( 1_{\widetilde U_g}\circ  (s_g)_{(A,U,\Pi)}\big)\big(1_{F_g(\widetilde X)}\circ \tau_g  \big)\\
&=\big( 1_{\widetilde U_g}\circ (1_\tau\circ  (s_g)_{(A,U,\Pi)}) \sigma^g_{(\tau,\tau_g)} \big)\big(1_{F_g(\widetilde X)}\circ \tau_g  \big)\\
&=\big( 1_{\widetilde U_g}\circ (s_g)_{(A,U,\Pi)} \big) \big( \tau_g\circ 1_{U^*_gF_g(X)U_g} \big)\big( F_g(\sigma_{(\tau,\tau_g)})\circ 1_{U_g} \big)\\
&=\big( \tau_g \circ 1_X \big)\big( 1_{F_g(\tau)}\circ  \theta_{g}\big)\big( F_g(\sigma_{(\tau,\tau_g)})\circ 1_{U_g}\big).
\end{align*}
The first equality follows by using the definition of $\theta^{(A,U,\Pi)}_{g}$ given in \eqref{def-thetag}, the second equality follows from \eqref{equivariant-ob-s0}, and the third one follows from the definition of $\sigma^g_{(\tau,\tau_g)}$.
\medbreak

(iii). The fact that $\Psi$ is an equivalence follows easily. A direct computation shows that 
$$ \Psi\big( ((X,\sigma),s)\ot ((Y,\tau),t) \big)= \Psi((X,\sigma),s) \ot \Psi((Y,\tau),t),$$
for any pair of objects $((X,\sigma),s), ((Y,\tau),t) \in \Zc(\Phi)^G$.
\epf

\end{document}